\documentclass[12pt]{amsart}
\usepackage{epsfig, graphics, psfrag, color,amsmath}
\usepackage[small,nohug,heads=littlevee]{diagrams}
\diagramstyle[labelstyle=\scriptstyle]

\def\blfootnote{\xdef\@thefnmark{}\@footnotetext}

\theoremstyle{plain}
\newtheorem{theorem}{Theorem}[section]
\newtheorem{corollary}[theorem]{Corollary}
\newtheorem{cor}[theorem]{Corollary}
\newtheorem{lemma}[theorem]{Lemma}

\newtheorem{proposition}[theorem]{Proposition}

\newtheorem*{namedtheorem}{\theoremname}
\newcommand{\theoremname}{testing}

\theoremstyle{definition}
\newtheorem{define}[theorem]{Definition}

\newtheorem*{remark}{Remark}

\psfrag{X1}[][bl]{$X_1$}
\psfrag{X2}[][bl]{$X_2$}
\psfrag{X3}[][bl]{$X_3$}
\psfrag{X4}[][bl]{$X_4$}
\psfrag{p23}[cc][Bl]{$(2,3)$}
\psfrag{p32}[cc][Bl]{$(3,2)$}
\psfrag{p14}[cc][Bl]{$(1,4)$}
\psfrag{p41}[cc][Bl]{$(4,1)$}
\psfrag{o1}[B][bl]{$\Omega_1$}
\psfrag{o1u}[B][bl]{$\Omega_1^\uparrow$}
\psfrag{o1d}[B][bl]{$\Omega_1^\downarrow$}
\psfrag{o2}[B][bl]{$\Omega_2$}
\psfrag{o2u}[B][bl]{$\Omega_2^\uparrow$}
\psfrag{o2d}[B][bl]{$\Omega_2^\downarrow$}
\psfrag{o3}[B][bl]{$\Omega_3$}
\psfrag{B}[][Bl]{$B$}
\psfrag{h}[][Bl]{$h$}
\psfrag{Q}[][Bl]{$g_r$}
\psfrag{K}[][Bl]{$B_r$}
\psfrag{W}[][Bl]{$h_t$}
\psfrag{bh}[][Bl]{$B_{\frac{1}{2}}$}
\psfrag{p}[][Bl]{$p$}
\psfrag{pp}[][Bl]{$P(M)$}
\psfrag{ppp}[][Bl]{$P(M)$}
\psfrag{psht}[][Bl]{$h_t$}
\psfrag{pshp}[][Bl]{$h'$}
\psfrag{hhp}[][Bl]{$h'\circ h^{-1}$}

\begin{document}


\title[An upper bound on Reidemeister moves]
{An upper bound on Reidemeister moves}

\author{Alexander Coward}
\address{Mathematics Department, University of California at Davis, CA 95616, USA}

\author{Marc Lackenby}
\address{Mathematical Institute, 24--29 St Giles', Oxford, OX1 3LB, England}


\begin{abstract}
We provide an explicit upper bound on the number of Reidemeister moves
required to pass between two diagrams of the same link. This leads to
a conceptually simple solution to the equivalence problem for links.
\end{abstract}

\maketitle

\newcommand{\mat}[2][cccc]{\left(\begin{array}{#1} #2\\
	\end{array}\right)}

\vspace{-20pt}

\section{Introduction}\label{sec:intro}

\let\thefootnote\relax\footnotetext{MSC (2010): 57M25, 57N10}
It is one of the most fundamental theorems in low-dimensional topology that any
two diagrams of a knot or link in $\mathbb{R}^3$ differ by a sequence of
Reidemeister moves, illustrated in Figure \ref{fig:reidemeister}.
Since Reidemeister's seminal paper \cite{RM} in 1926, it has been speculated as to whether there is an explicit upper bound for the number of moves that are needed, as a function of the number
of crossings in the initial and terminal diagrams. See \cite{RMbook} or page 15 of \cite{adams} for example. 
In a celebrated paper \cite{HL}, Hass and Lagarias provided such a bound when the link in question is the unknot, this bound being an exponential
function of the number of crossings in the diagrams. 
In this paper we answer the general question with the following theorem, which applies to all knots and links.


\begin{theorem} \label{maintheorem}
Let $D_1$ and $D_2$ be connected diagrams for some knot or link in $\mathbb{R}^3$,
and let $n$ be the sum of their crossing numbers. Then $D_2$ may be
obtained from $D_1$ by a sequence of at most $\exp^{\left( c^{n} \right) }(n)$
Reidemeister moves, where $c = 10^{1,000,000}$.
\end{theorem}


\begin{figure}[h]
\centering
\includegraphics{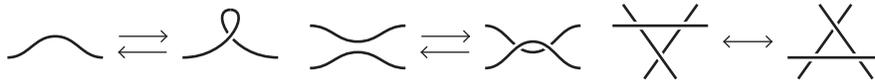}
\caption{Reidemeister moves} \label{fig:reidemeister}
\end{figure}

Here, a \emph{link diagram} is a 4-valent
graph embedded in $\mathbb{R}^2$ with each vertex decorated with `over' and `under' crossing information.  If a knot or link is oriented, then its diagrams have directed edges that agree with this orientation. Theorem \ref{maintheorem} applies to link diagrams with or without orientation. We view two diagrams as the same if their decorated (oriented) graphs are ambient isotopic in $\mathbb{R}^2$.
The function $\exp(x)$ is the exponential function $2^x$, and $\exp^{(r)}(x)$ means iterate
this function $r$ times. Thus, $\exp^{\left( c^{n} \right) }(n)$ is shorthand for a tower
of 2s with an $n$ at the top, the height of the tower being $c^n$. This upper bound is very large indeed, but it is explicit and computable. 
This therefore leads to a conceptually very simple algorithm for
solving the equivalence problem for links. Given two links $L_1$ and $L_2$
with connected diagrams $D_1$ and $D_2$, one can decide whether they are the same
link as follows. Let $n$ be the sum of their crossing numbers.
Apply all possible sequences of Reidemeister moves to $D_1$ of length
at most $\exp^{\left (c^{n} \right ) }(n)$. If $L_1$ and $L_2$ are equivalent
links, one of these diagrams will be ambient isotopic to $D_2$, and this can readily be
determined. On the other hand, if $L_1$ and $L_2$ are inequivalent, then
none of these diagrams will be isotopic to $D_2$.

It is trivial that there is some function $F \colon {\mathbb N} \times {\mathbb N}
\rightarrow {\mathbb N}$ such that any two connected diagrams $D_1$ and $D_2$
of a link with $n_1$ and $n_2$ crossings differ by a sequence of
at most $F(n_1,n_2)$ Reidemeister moves. This just follows from
the fact that there are only finitely many connected diagrams with a given number
of crossings, and Reidemeister's theorem. However, the existence
of a computable function $F$ is a much stronger statement,
as the following simple theorem demonstrates.

\begin{theorem} \label{computablefunction}
The following are equivalent:
\begin{enumerate}
 \item There is a computable function $F \colon {\mathbb N} \times {\mathbb N} \rightarrow {\mathbb N}$
such that for any two connected diagrams $D_1$ and $D_2$ of a link with $n_1$ and
$n_2$ crossings, there is a sequence of at most $F(n_1, n_2)$
Reidemeister moves that takes $D_1$ to $D_2$.
\item There is an algorithm to solve the equivalence problem for links. In other words,
there is an algorithm that takes as input two link diagrams
and determines whether or not they represent equivalent links.
\end{enumerate}
\end{theorem}

The proof of (1) $\Rightarrow$ (2) is described above. For
(2) $\Rightarrow$ (1), we need to produce an algorithm that, given natural numbers
$n_1$ and $n_2$, computes a natural number $F(n_1,n_2)$ with the required properties.
To do this, the computer enumerates all connected link diagrams with at most $\max \{n_1,n_2\}$
crossings. Then using the hypothesised algorithm, the computer
arranges these into groups according to their link type.
Then, the computer searches for Reidemeister moves relating all diagrams of
each type. Such sequences of moves exist by Reidemeister's theorem.
Hence, eventually, an upper bound on the number of moves
will be computed.

We call a function $F \colon {\mathbb N} \times {\mathbb N} \rightarrow {\mathbb N}$ a \emph{Reidemeister move function} for a link $L$ (which may be oriented or unoriented) if for any two connected diagrams $D_1$ and $D_2$ of $L$ with $n_1$ and
$n_2$ crossings, there is a sequence of at most $F(n_1, n_2)$
Reidemeister moves that takes $D_1$ to $D_2$. Thus, Theorem \ref{maintheorem} gives a Reidemeister move function that applies to all links.

The equivalence problem for links was solved by Haken \cite{hakennormal} and Hemion \cite{hemionart} in the 1960s and 1970s. In fact, an alternative solution to the homeomorphism problem for hyperbolic link complements, and hence the equivalence problem for hyperbolic knots, was given by Dahmani and Groves \cite{altrecog}, based on work of Sela \cite{sela}. 
Thus, the existence of a computable function $F$ as above was already known.
However, this is not `explicit'. We leave this term undefined, but we
hope that the reader will agree that the function provided by Theorem
\ref{maintheorem} is explicit, whereas that provided by Theorem
\ref{computablefunction} is not. In particular, the bound in Theorem \ref{maintheorem}
is primitive recursive, whereas this is not obviously true of the bound provided
by Theorem \ref{computablefunction}.

Theorem \ref{maintheorem} is proved using
triangulations and Pachner moves. Much as any two diagrams of a link
are related by a sequence of Reidemeister moves, any two triangulations
of a PL manifold are related by a sequence of Pachner moves \cite{pach1, pach2}.
The key theorem we use is an adaptation of a result of Mijatovi\'c  \cite{alexknot}, who
provides an explicit upper bound on the number of Pachner moves required
to pass between two triangulations of a knot exterior. 
However, the translation from a bound on Pachner moves to
a bound on Reidemeister moves is not a straightforward one. 

Starting with two connected diagrams $D_1$ and $D_2$ for a link, we pick
embeddings $L_1$ and $L_2$ of the link in $\mathbb{R}^3$ which project to these
diagrams and which both lie in some convex 3-ball. We use $D_1$ and
$D_2$ to build two triangulations for the link
exterior in this 3-ball. Using an adaptation of Mijatovi\'c's theorem, these are
related by a bounded number of Pachner moves. These induce an explicit PL homeomorphism between the
triangulated link exteriors. This extends to a homeomorphism of the
3-ball sending $L_1$ to $L_2$. The bound on the number of Pachner moves provides 
some control on this homeomorphism. We then apply Alexander's trick to specify an ambient isotopy sending $L_1$ to $L_2$, and this ambient isotopy is again of controlled complexity in a certain sense. The bound on the complexity of this ambient isotopy is ultimately what provides us with our upper bound for the number of Reidemeister moves required to pass between $D_1$ and $D_2$.

A significant complication arises from the fact that Mijatovi\'c's theorem is not quite sufficient for our purposes. His result provides an explicit upper bound on the number of Pachner moves required to pass between two triangulations of a link exterior, up to a homeomorphism that is isotopic to the identity in the hyperbolic pieces of the link's JSJ decomposition. This is insufficient, for two reasons. Firstly, many link exteriors have Seifert fibred pieces in their JSJ decomposition, and so the resulting homeomorphism of the link exterior may act non-trivially on the boundary, in which case it may not extend to a homeomorphism of the 3-ball. But Mijatovi\'c's theorem is not sufficient for our purposes even in the case of hyperbolic knots, because it only provides a bound on the number of Pachner moves {\it up to ambient isotopy}, and yet it is exactly an explicit ambient isotopy that we are aiming to construct.

It is therefore necessary for us to prove a strengthened version of
Mijatovi\'c's theorem. This takes some effort, because it requires us
to go through his proofs, originally exposed in \cite{alexs3}, \cite{alexff}, \cite{alexsf} and \cite{alexknot}, and adapt them. We have thus endeavoured to give an
accessible outline of Mijatovi\'c's work, highlighting the
similarities and differences between his methods and ours. The
following is our result.

\begin{theorem} \label{mijatovicmodified}
Let $M$ be a compact orientable irreducible 3-manifold
with boundary a non-empty collection of tori.
Suppose that the closure of each component of the complement of the characteristic
submanifold of $M$ satisfies at least one of the following conditions:
\begin{itemize}
\item it does not fibre over the circle; or
\item it is not a surface semi-bundle; or
\item it has at least two boundary components.
\end{itemize}
Let $T_{\partial M}$ be any triangulation of $\partial M$. Then there is a
triangulation ${\mathcal T}_{\rm can}'$ of $M$ with the following properties. Its restriction
to $\partial M$ equals $T_{\partial M}$. Further, if $T$ is any triangulation
of $M$ with $t$ tetrahedra, such that the restriction of $T$ to $\partial M$ also equals $T_{\partial M}$,
then there is a sequence of at most $\exp^{(a^{t})}(t)$
interior Pachner moves,
followed by a homeomorphism of $M$ that is the identity on $\partial M$,
taking $T$ to ${\mathcal T}_{\rm can}'$. This homeomorphism is isotopic to one equal to the identity on the complement of the characteristic submanifold of $M$. Here, $a = 2^{162}$.
\end{theorem}

By an \emph{interior Pachner move}, we mean a Pachner move that does not affect the triangulation of the boundary of $M$. See Section \ref{sec:pltheory} for a precise definition. A \emph{surface semi-bundle} is a compact orientable 3-manifold obtained from two $I$-bundles over non-orientable surfaces by identifying their horizontal boundaries via a homeomorphism.

We use the phrase `canonical triangulation' for ${\mathcal T}_{\rm can}'$. 
The word `canonical' needs to be used with some caution, because it depends 
on the given triangulation $T_{\partial M}$ of $\partial M$, and arbitrary choices \emph{are} made in its construction.
It is built out of a collection of surfaces in $M$, which is very nearly a hierarchy.

The above theorem can of course be used to bound the number of interior Pachner moves required to pass between two triangulations of $M$ that are equal on $\partial M$. We call a function $F \colon {\mathbb N} \times {\mathbb N} \rightarrow {\mathbb N}$ a \emph{Pachner move function} for a compact 3-manifold $M$ if for any two triangulations $T_1$ and $T_2$ for $M$, equal on $\partial M$ and with at most $n_1$ and
$n_2$ tetrahedra, there is a sequence of at most $F(n_1, n_2)$
interior Pachner moves, followed by a homeomorphism of $M$ that is the identity on $\partial M$, that takes $T_1$ to $T_2$. 
Thus Theorem \ref{mijatovicmodified} implies that if $M$ is a 3-manifold satisfying the hypotheses of  the theorem, then $(n_1,n_2) \mapsto \exp^{(a^{n_1})}(n_1) + \exp^{(a^{n_2})}(n_2)$ is a Pachner move function for $M$ where $a = 2^{162}$. It is in this way that we will apply Theorem \ref{mijatovicmodified}.


The bound on
Pachner moves in Theorem \ref{mijatovicmodified}  gives the bound on Reidemeister moves in
Theorem  \ref{maintheorem}, via the following result.


\begin{theorem}\label{fromptor}
Let $L$ be an oriented non-split link in $S^3$. Suppose that $P_L \colon {\mathbb N} \times {\mathbb N} \rightarrow {\mathbb N}$ is a Pachner move function for the exterior of $L$ in $S^3$. Then $R_L \colon {\mathbb N} \times {\mathbb N} \rightarrow {\mathbb N}$ given by  $$R_L(n_1,n_2) = \exp^{\left( 2 \right) }({400(P_L(2^{14}(n_1+n_2),2^{14}(n_1+n_2)) + 2^{16}(n_1+n_2))} )$$ is a Reidemeister move function for $L$.
\end{theorem}

An outline of this paper is as follows. In Section 2,
we introduce some of the basic PL machinery that we require.
In Section 3, we explain how to work with triangulations of the 3-ball rather than
the 3-sphere. In Section 4, we explain how to construct, starting with a connected diagram of a link $L$,
an explicit triangulation $T$ of a 3-ball that contains $L$
as a subcomplex and also possesses some other properties that we will require. In Section 5, we review Alexander's trick and use it to show how one may pass from a homeomorphism of the 3-ball, of known complexity, sending $L_1$ to $L_2$, to an ambient isotopy sending $L_1$ to $L_2$ also of known complexity.  This ambient isotopy induces a 1-parameter family of link projections interpolating between $D_1$ and $D_2$. However these link projections may not be diagrams in the usual sense. Thus in Section 6, we show how one may pass from a 1-parameter family of link projections to a sequence of diagrams each related to the next by a single Reidemeister move. In Section 7, Theorem \ref{fromptor} is proved, leading quickly to Theorem \ref{maintheorem}.
In Sections 8 and 9, we state Mijatovi\'c's result, compare it with Theorem \ref{mijatovicmodified}, and
summarise his proof in the case of simple and Seifert fibred 3-manifolds. In Sections 10, 11, 12 and 13, we
give our proof of Theorem \ref{mijatovicmodified}.

\section{Piecewise linear theory}\label{sec:pltheory}

Our main theorem is proved using piecewise linear techniques. We therefore start by
giving precise definitions from this theory.

For $n \in {\mathbb N}$, let $\Delta^n$ be the standard $n$-simplex.
An \emph{abstract $\Delta$-complex} $K$ is:
\begin{itemize}
\item an indexing set $A(n)$, for each natural number $n$;
\item a copy $\Delta_\alpha^n$ of the standard $n$-simplex for each $\alpha \in A(n)$;
\item for each $\Delta_\alpha^n$, where $n > 0$,
an affine homeomorphism between each $(n-1)$-dimensional face of $\Delta_\alpha^n$
and some $\Delta_\beta^{n-1}$ for $\beta \in A(n-1)$.
\end{itemize}

The term $\Delta$-complex was also used by Hatcher in \cite{hatcher}.
However, his use of the term was slightly different. He ordered
the vertices of each simplex $\Delta_\alpha^n$ and
required that each affine homeomorphism from a face of
$\Delta_\alpha^n$ to $\Delta_\beta^{n-1}$ preserved the
ordering of the vertices. We will not make that requirement.

The basic example of an abstract $\Delta$-complex is a simplicial
complex.

One forms the \emph{underlying space} $|K|$ of an abstract $\Delta$-complex $K$ by starting with the
disjoint union of the simplices, and identifying each face of
each simplex $\Delta_\alpha^n$ with the corresponding simplex
$\Delta_\beta^{n-1}$. We shall use the term \emph{$\Delta$-complex}
to denote either an abstract $\Delta$-complex or its underlying space.

A \emph{combinatorial isomorphism} between abstract $\Delta$-complexes
$K$ and $K'$ is, for each natural number $n$, a bijection between their
indexing sets $A(n)$ and $A'(n)$, together with a collection of affine homeomorphisms $\Delta_\alpha^n \rightarrow \Delta_{\alpha'}^n$ for each $\alpha \in A(n)$
associated with $\alpha' \in A'(n)$, such that for each $(n-1)$-dimensional face $F$ of
$\Delta^n_\alpha$ and $F'$ of $\Delta^n_{\alpha'}$ that correspond, the following
diagram commutes:

\begin{diagram}
F &\rTo &F'\\
\dTo_{\cong} & &\dTo_{\cong} \\
\Delta^{n-1}_\beta &\rTo & \Delta^{n-1}_{\beta'}
\end{diagram}
Here, the vertical maps are the affine homeomorphisms in the
definition of a $\Delta$-complex. The top horizontal map
is the restriction of the affine homeomorphism
$\Delta^n_\alpha \rightarrow \Delta^n_{\alpha'}$.
The bottom horizontal map is the affine
homeomorphism arising from the bijection between $A(n-1)$ and
$A'(n-1)$. Put simply, a combinatorial isomorphism is a bijection
that preserves all the structure of the complex.

The combinatorial isomorphism between $K$ and $K'$
determines a homeomorphism $|K| \rightarrow |K'|$ such that the interior of each simplex of $|K|$
is sent to the interior of a simplex of $|K'|$ via an affine homeomorphism.
We shall also term this homeomorphism a \emph{combinatorial isomorphism}.

A \emph{subdivision} of a $\Delta$-complex $K$ is a $\Delta$-complex $K'$
together with a homeomorphism $h \colon |K'| \rightarrow |K|$
such that each simplex $\sigma$ of $|K'|$ is mapped into a simplex
of $|K|$ and the restriction of $h$ to $\sigma$ is affine.


A \emph{triangulation} $T$ of a space $M$ is a $\Delta$-complex $K$
together with a homeomorphism $M \rightarrow |K|$. (When $M$ already
has a PL structure, we insist that this homeomorphism is PL.) We let $|T|$
denote $|K|$. When $K$ is
actually a simplicial complex, we term this a \emph{genuine triangulation}.
(In traditional PL theory, what we term a `genuine triangulation'
is just called a `triangulation', and only these are considered.
However, it has become standard practice, particularly in low-dimensional
topology, to work with $\Delta$-complexes rather than restrict to
simplicial complexes.)


When $K'$ is
a subdivision of $K$, there is an induced homeomorphism $M \rightarrow |K'|$
which is the \emph{subdivided triangulation}.



Two triangulations $h \colon M \rightarrow |K|$ and $h' \colon M \rightarrow |K'|$
of a space $M$ are \emph{equal} if there is a combinatorial isomorphism
$c \colon |K| \rightarrow |K'|$ such that the following diagram commutes:

\begin{diagram}
& & M & & \\
& \ldTo_h & & \rdTo_{h'} \\
|K| && \rTo^{c} && |K'|
\end{diagram}

This is a rather restrictive condition. If we are only given a combinatorial isomorphism
$c \colon |K| \rightarrow |K'|$, there is an associated homeomorphism
$h'^{-1} c h \colon M \rightarrow M$
which \emph{realises} this combinatorial
isomorphism. There is no guarantee that this is the identity.
We say that the two triangulations are \emph{isotopy-equivalent}
if this homeomorphism is isotopic to the identity on $M$.
We also say that two triangulations are \emph{homeomorphism-equivalent}
if their $\Delta$-complexes are combinatorially isomorphic,
but with no restriction on the commutativity of the above diagram.
Thus, homeomorphism-equivalence of triangulations is essentially the same concept as combinatorial
isomorphism, but it is useful to have two different terms.

It is helpful to consider an example. Suppose that $M$ is a space with
infinite mapping class group. Then it is immediate that if it admits a
triangulation with $N$ simplices, it admits infinitely many such
triangulations that are homeomorphism-equivalent but isotopy-inequivalent, as follows.
Starting with one such triangulation $h \colon M \rightarrow |K|$,
one may apply non-isotopic homeomorphisms $\phi_i \colon M \rightarrow M$ ($i \in {\mathbb N}$)
to obtain new triangulations $h\phi_i$. Since $K$ is finite, there
are only finitely many combinatorial isomorphisms $|K| \rightarrow |K|$,
and these realise only finitely many isotopy classes of homeomorphisms
$M \rightarrow M$. Hence, infinitely many of these triangulations of
$M$ are distinct, even up to isotopy-equivalence. 
This lack of finiteness
is an important phenomenon that occurs often in the situations we will
examine. For example, the torus has infinite mapping class group,
as does the exterior of every satellite knot.

We will be concerned with triangulations of manifolds. In this case,
there is a well known way to vary a triangulation, via Pachner moves,
which are defined as follows. Let $T$ be a triangulation of an $n$-manifold $M$.
Let $\Delta$ be a standard $(n+1)$-simplex.
Let $D$ be a non-empty subset of $\partial \Delta$, consisting of $i<n+2$ faces
of dimension $n$. Suppose that some subcomplex of $T$ has interior that is
combinatorially isomorphic to the interior of $D$. (The typical case
is when the subcomplex itself is combinatorially isomorphic to $D$,
but we wish to permit the possibility that it may be obtained from
a copy of $D$ by identifying simplices on the boundary.) Then the operation of
removing from $T$ the copy of the interior of $D$ and inserting the interior of
${\rm cl}(\partial \Delta - D)$ is an \emph{interior
Pachner move}. This is known as an $(i,n+2-i)$ move,
since it replaces $i$ $n$-simplices with $n+2-i$ ones. See Figure \ref{explicitmoves}.

\begin{figure}[h]
\centering
\includegraphics{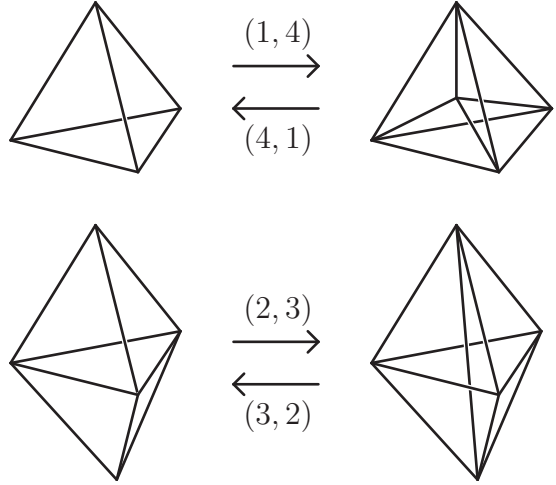}
\label{explicitmoves}
\caption{Three-dimensional Pachner moves}
\end{figure}

If triangulations $T$ and $T'$ are related by an interior Pachner move, then they have a common
subdivision. More precisely, one can also consider the triangulation
$T''$ that is obtained by replacing the interior of $D$ by the
cone on $\partial D$. It is possible to realize this cone as
a subdivision of both $D$ and ${\rm cl}(\partial \Delta - D)$, in a way that is canonical up to isotopy equivalence.
We therefore obtain homeomorphisms $h \colon |T''| \rightarrow |T|$ and
$h' \colon |T''| \rightarrow |T'|$, and hence a homeomorphism
$h' \circ h^{-1} \colon |T| \rightarrow |T'|$. See Figure \ref{pachsubdnew}.

\begin{figure}[h]

\psfrag{T}[][bl]{$T$}
\psfrag{Y}[][bl]{$T'$}
\psfrag{U}[][bl]{$T''$}

\centering
\includegraphics{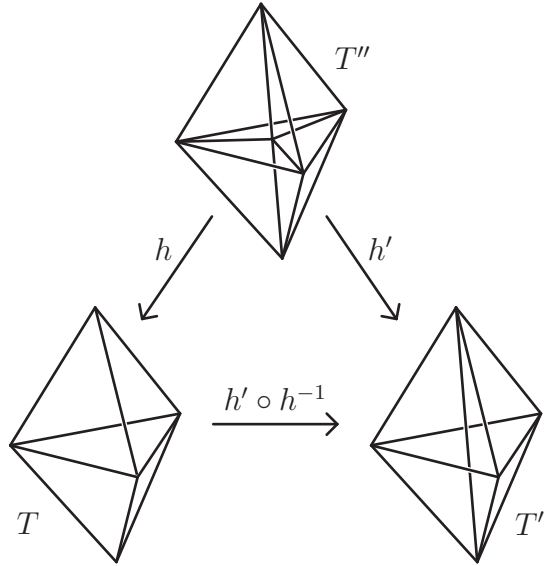}
\caption{$T''$ is a common subdivision of both $T$ and $T'$}
\label{pachsubdnew}
\end{figure}

If the $n$-manifold $M$ has non-empty boundary, Pachner moves do not
affect the triangulation of $\partial M$. Thus, it is necessary to
introduce a related move, which we call a \emph{boundary Pachner move}.
Here, one starts with a standard $n$-simplex $\Delta$
and a non-empty subset $D$ of $\partial \Delta$ consisting of $i < n+1$ faces
of dimension $n-1$. One finds a subset of $\partial M$ whose interior is combinatorially
isomorphic to the interior of $D$, and one attaches $\Delta$ to it along $D$. Alternatively, one
performs the reverse of this procedure. This operation has the effect of changing the triangulation on $\partial M$ by
an interior Pachner move.

%

A \emph{Pachner move} refers to an interior or a boundary Pachner move.

The following key result about Pachner moves is due to Pachner \cite{pach1,pach2}.
See also \cite{lickpach}.


\begin{theorem}
 Let $T$ and $T'$ be two triangulations of a closed PL manifold $M$.
Then, up to homeomorphism-equivalence, there is a finite sequence of Pachner moves that takes $T$ to $T'$.
\end{theorem}

One might then wish to determine a bound on this number of moves,
solely in terms of the number $t$ and $t'$ of top-dimensional
simplices in $T$ and $T'$. It trivially holds that there is a function
$F \colon {\mathbb N} \times {\mathbb N} \rightarrow {\mathbb N}$
such that any two triangulations $T$ and $T'$ differ by a
sequence of at most $F(t, t')$ Pachner moves, up to homeomorphism-equivalence.
However, any useful bound is impossible in general, as the following
theorem demonstrates. This is proved in much the same way
as Theorem \ref{computablefunction}.


\begin{theorem}
Let $M$ be a compact orientable PL $m$-manifold, where $m \leq 4$. Then the following are equivalent:
\begin{enumerate}
 \item There is a computable function $F \colon {\mathbb N} \times {\mathbb N} \rightarrow {\mathbb N}$
such that for any two triangulations $T_1$ and $T_2$ of $M$, with $t_1$ and $t_2$ $m$-simplices
respectively, there is a sequence of at most $F(t_1,t_2)$
Pachner moves that takes $T_1$ to a triangulation homeomorphism-equivalent to $T_2$.
\item There is an algorithm to recognize $M$ among all triangulated compact PL $m$-manifolds. In other words,
there is an algorithm that takes as input the triangulation of some compact $m$-manifold
and determines whether or not this manifold is PL homeomorphic to $M$.
\end{enumerate}
\end{theorem}

Note that, by a theorem of Novikov, 
there are closed $m$-dimensional manifolds which cannot be recognized among the set of all triangulated
PL $m$-manifolds, provided $m \geq 4$. Indeed, the $m$-sphere falls in this class when $m \geq 5$.
(See the appendix to \cite{nabutovsky} for example.) An explicit example of such a manifold
when $m = 4$ is the connected sum of 16 copies of $S^2 \times S^2$ \cite{shtanko}.

One might wonder why it is necessary to assume that $m \leq 4$ in the above result. This is required in 
the proof of $(2) \Rightarrow (1)$. Here, one starts with a recognition algorithm for $M$ and from this, 
one constructs an algorithm to compute $F(t_1, t_2)$ for any positive integers $t_1$ and $t_2$. To do this, 
one constructs all spaces obtained from $\max \{ t_1, t_2 \}$ $m$-simplices by identifying 
their $(m-1)$-dimensional faces in pairs. But one must then discard all spaces that do not form an $m$-manifold, 
and to do this, a recognition algorithm for the $(m-1)$-sphere is required. This is only known to exist for $m \leq 4$ \cite{thompson}.
Then, once one has constructed this collection of triangulated $m$-manifolds, one then applies the recognition 
algorithm for $M$ to discard the triangulations of manifolds not PL homeomorphic to $M$. The resulting 
triangulations of $M$ are all related by sequences of Pachner moves, which one can eventually construct. 
One defines $F(t_1, t_2)$ to be the maximum length of each such sequence of Pachner moves.

The solution to the recognition problem for Haken 3-manifolds was established
by Haken \cite{hakennormal} and Hemion \cite{hemionart}. Hence, for 3-manifolds $M$ in this class,
there are computable functions $F$ as above. Moreover, Mijatovi\'c in
\cite{alexknot} provided an explicit and easily computed function $F$ for
a large class of Haken 3-manifolds that includes the exteriors of all
non-split links in the 3-sphere. It is Mijatovi\'c's explicit upper bound, and the technology behind it,
that is the key to this paper.

\section{Pachner moves on punctured 3-manifolds}\label{sec:3ball}

In this section, we will deal with 3-manifolds which may have a
2-sphere boundary component. Our goal is to prove the following theorem.

\begin{theorem} \label{3balltriangulations}
Let $T_1$ and $T_2$ be triangulations of a 3-ball $B$.
Suppose that in a collar neighbourhood of $\partial B$, these
triangulations are equal, and are the same standard triangulation
of $\partial B \times I$.
Let $t_1$ and $t_2$ be the number of 3-simplices in $T_1$ and $T_2$ respectively.
Suppose that links $L_1$ and $L_2$ are subcomplexes of $T_1$ and $T_2$
respectively, and that they have triangulated neighbourhoods $N(L_1)$
and $N(L_2)$ that are combinatorially isomorphic. Suppose that the homeomorphism
from $N(L_1)$ to $N(L_2)$ that realises this combinatorial isomorphism
preserves the longitudal slope of each boundary component. Suppose also that $L_1$
and $L_2$ are ambient isotopic, and that $P \colon {\mathbb N} \times {\mathbb N} \rightarrow {\mathbb N}$ is a Pachner move function for the exterior of this ambient isotopy class of link in $S^3$. Then there exists a sequence of at
most $100 P(t_1,t_2) + 100 (t_1 + t_2)$ Pachner moves,
followed by a combinatorial isomorphism,
that takes $T_1$ to $T_2$, and $N(L_1)$ to $N(L_2)$. Furthermore,
none of the Pachner moves affect $N(L_1)$ or $\partial B$,
and the combinatorial isomorphism restricts to the identity on $\partial B$.
\end{theorem}

We say that the \emph{standard triangulation}
of the 2-sphere $S$ is the boundary of a 3-simplex.

Let $T_{S \times I}$ be the following triangulation of $S \times I$.
Give each component of $S \times \partial I$ the same standard triangulation.
For each edge $e$ of this triangulation, insert $e \times I$ into
$S \times I$. Place a vertex in the interior of this rectangle, and
cone off the rectangle from this vertex. These rectangles divide $S \times I$
into a collection of balls. Insert a vertex in the interior of each
such ball and cone off the ball from this vertex. The resulting
triangulation is $T_{S \times I}$, which we term the \emph{standard triangulation}
of $S \times I$.

Theorem \ref{3balltriangulations} will be a consequence of the following result.

\begin{theorem} \label{2sphereboundary}
Let $M$ be a compact orientable 3-manifold, and let
$S$ be a 2-sphere boundary component of $M$. Suppose that the 3-manifold
that results from attaching a 3-ball to $S$  has $P \colon {\mathbb N} \times {\mathbb N} \rightarrow {\mathbb N}$ as a Pachner move function. Let $T_1$ and $T_2$ be triangulations of $M$
with at most $t_1$ and $t_2$ tetrahedra respectively.
Suppose that $T_1$ and $T_2$ are equal on $\partial M$, and
that in a collar neighbourhood of $S$, $T_1$ and $T_2$ are equal
and combinatorially isomorphic to $T_{S \times I}$. Then there is
a sequence of at most
$$100 P(t_1,t_2) + 100 (t_1 + t_2).$$
interior Pachner moves, followed by
a homeomorphism that is the identity on $\partial M$, taking
$T_1$ to $T_2$. 
\end{theorem}
%

We shall need some terminology before embarking on the proof of Theorem \ref{2sphereboundary}.

\begin{define} Let $\hat M$ be a compact orientable 3-manifold, let $T$ 
be a triangulation  $\phi \colon
\hat M \rightarrow |T|$ of $\hat M$, and let 
$x$ be a point in $\hat M$ that is disjoint from the image under $\phi^{-1}$
of the 2-skeleton.
We define the following \emph{puncturing operation}. Remove the interior of
the 3-simplex containing $\phi(x)$, and insert $T_{S \times I}$.
Let $P(\hat M,T,x)$ be the resulting triangulation of the manifold
$\hat M - {\rm int}(N(x))$.
\end{define}

Let $T_1$ and $T_2$ be triangulations of the same 3-manifold $\hat M$. Suppose that
they differ by a sequence of $N$ interior Pachner moves
$$T_1 = T^0 \rightsquigarrow T^1 \rightsquigarrow \dots
\rightsquigarrow T^N = T_2.$$
For each relevant integer $i$, let $K^i$ be the $\Delta$-complex associated with
$T^i$, and let $h_i \colon |K^{i-1}| \rightarrow |K^i|$
be the homeomorphism resulting from the Pachner move. Let $x$ be a point in $\hat M$
disjoint from the inverse image in $\hat M$ of the 2-skeleton of
$|K^0|, |K^1|, \dots$ and  $|K^N|$. 


\begin{lemma} \label{puncturingbound}
The triangulations $P(\hat M,T^0,x)$ and $P(\hat M,T^N,x)$
differ, up to ambient isotopy fixed on $\partial \hat M \cup \partial N(x)$, 
by a sequence of at most $100N$ interior Pachner moves.
\end{lemma}

\begin{proof} It clearly suffices to prove that for each $i$,
$P(\hat M,T^i,x)$ and $P(\hat M,T^{i+1},x)$ differ by a sequence
of at most $100$ Pachner moves. An example is shown in Figure \ref{fig:puncpachner},
where $T^i$ and $T^{i+1}$ differ by a $(1,4)$-move. (Note that, for clarity, not all
simplices in these triangulations are drawn.) Since there are
only 4 types of move, it is clear that there
is a universal constant $k$ such that
$P(\hat M,T^i,x)$ and $P(\hat M,T^{i+1},x)$ differ by at
most $k$ interior Pachner moves. An elementary calculation,
which is omitted, proves that $k = 100$ suffices.
\end{proof}

\begin{figure}[h]
\centering
\psfrag{d}[][Bl]{$P(\hat M,T^i,x)$}
\psfrag{e}[][Bl]{$P(\hat M,T^{i+1},x)$}

\includegraphics{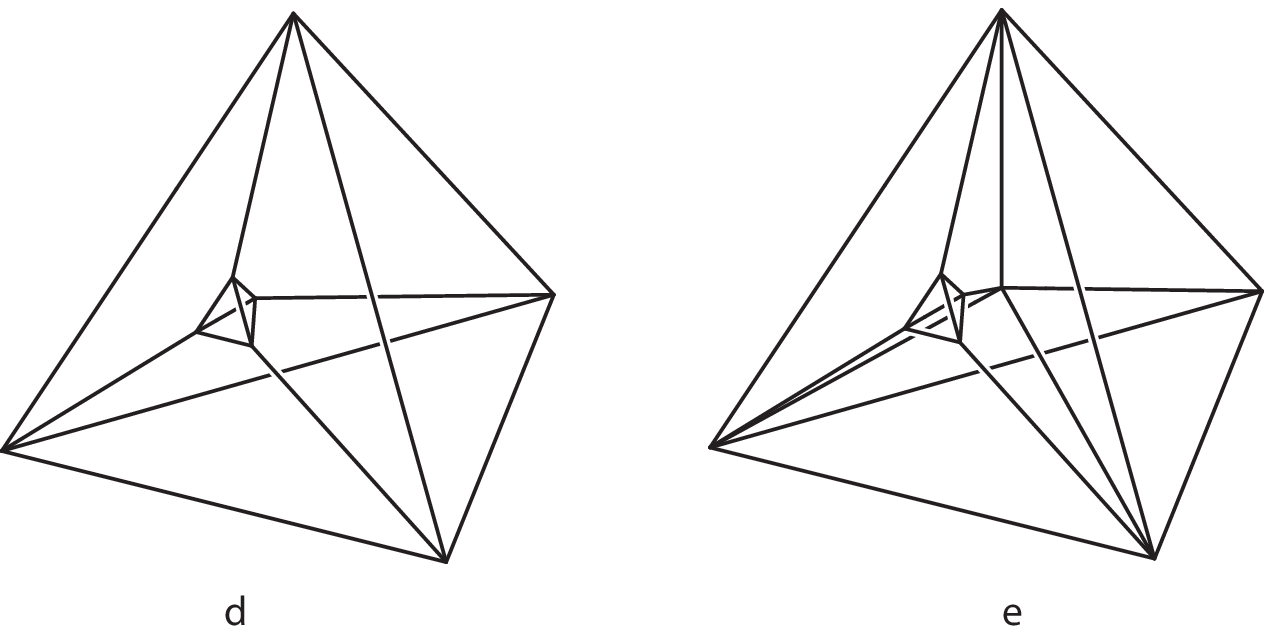}

\caption{} \label{fig:puncpachner}
\end{figure}

\begin{lemma} \label{basepointchange}
Let $T$ be a triangulation of a compact orientable
3-manifold $\hat M$ with $t$ tetrahedra. Let $x$ and $x'$ be points in $\hat M$
that are in the complement of the 2-skeleton. Then, $P(\hat M,T,x)$ and
$P(\hat M,T,x')$ are triangulations of $\hat M - {\rm int}(N(x))$
and $\hat M - {\rm int}(N(x'))$ respectively, from which $\partial N(x)$
and $\partial N(x')$ inherit triangulations. Let $h \colon \partial N(x)
\rightarrow \partial N(x')$ be a simplicial isomorphism.
Then, $P(\hat M,T,x)$ and $P(\hat M,T,x')$ differ by a
sequence of at most $100t$ interior Pachner moves, followed
by a homeomorphism $\hat M - {\rm int}(N(x)) \rightarrow \hat M - {\rm int}(N(x'))$.
We may arrange that this homeomorphism equals $h$ on $\partial N(x)$,
equals the identity on $\partial \hat M$, and extends to a homeomorphism
$\hat M \rightarrow \hat M$ that is isotopic to the identity.
\end{lemma}

\begin{proof} Suppose that $x$ and $x'$ lie in distinct 3-simplices $\Delta$
and $\Delta'$ that share a face. Let $\Delta \cup \Delta'$ denote the
3-ball obtained by gluing $\Delta$ and $\Delta'$ along this face.
There may be other identifications on the boundary of this ball,
but we do not yet make these. The triangulations $P(\Delta,\Delta, x) \cup \Delta'$
and $\Delta \cup P(\Delta', \Delta', x')$ are triangulations
of $(\Delta \cup \Delta') - {\rm int}(N(x))$ and 
$(\Delta \cup \Delta') - {\rm int}(N(x'))$ respectively.
Pick a homeomorphism $\phi \colon (\Delta \cup \Delta') - {\rm int}(N(x))
\rightarrow (\Delta \cup \Delta') - {\rm int}(N(x'))$ that
restricts to $h$ on $\partial N(x)$, equals the identity on
$\partial (\Delta \cup \Delta')$ and extends to a homeomorphism
$(\Delta \cup \Delta') \rightarrow (\Delta \cup \Delta')$ that
is isotopic, relative to its boundary, to the identity.
Then, using $\phi$ to pull back the triangulation of $(\Delta \cup \Delta') - {\rm int}(N(x'))$, 
we obtain another triangulation of
$(\Delta \cup \Delta') - {\rm int}(N(x))$. It is clear that there is some number $k$
of interior Pachner moves, followed by an ambient
isotopy that equals the identity on $\partial (\Delta \cup \Delta') \cup \partial N(x)$ 
taking one to the other. One may easily check that
$k = 100$ suffices. Since only interior Pachner moves were used,
any further identifications on the boundary of $\Delta \cup \Delta'$
do not affect this argument.

Now suppose that $x$ and $x'$ are arbitrary points in $\hat M$ disjoint from the 2-skeleton.
There is a sequence of points
$x = x_0, \dots, x_r = x'$ of points disjoint from the 2-skeleton, where $r \leq t$,
and for each $i$, $x_i$ and $x_{i+1}$ lie in distinct 3-simplices that share
a face. Thus, at most $kr \leq 100t$ Pachner moves, followed by a
homeomorphism as described in the lemma, suffice to
take $P(\hat M,T,x)$ to $P(\hat M,T,x')$.
\end{proof}

\begin{proof}[Proof of Theorem \ref{2sphereboundary}]
Let $\hat M$ be the result of attaching a
3-ball to $M$ along $S$. Let $\hat T_1$ and $\hat T_2$ be the result of
removing the copy of $T_{S \times I}$ from $T_1$ and $T_2$,
and inserting a 3-simplex, giving triangulations $\phi_1 \colon \hat M \rightarrow |\hat T_1|$
and $\phi_2 \colon \hat M \rightarrow |\hat T_2|$.
Thus, $T_1 = P(\hat M, \hat T_1, y)$
and $T_2 = P(\hat M, \hat T_2, y)$, where $y$ is a point in the 
interior of the newly attached 3-ball.
By assumption, there is a sequence of
$N \leq P(t_1,t_2)$
interior Pachner moves
$$\hat T_1 = T^0 \rightsquigarrow T^1 \rightsquigarrow \dots
\rightsquigarrow T^N,$$
where $T^N$ is homeomorphism-equivalent to $\hat T_2$, 
via a homeomorphism $h \colon \hat M \rightarrow \hat M$
that is the identity on $\partial \hat M$.
For each relevant integer $i$, let $K^i$ be the $\Delta$-complex associated with
$T^i$, and let $h_i \colon |K^{i-1}| \rightarrow |K^i|$
be the homeomorphism resulting from the Pachner move. Let $x$ be a point in $\hat M$
disjoint from the inverse image of the 2-skeletons of every
$|K^i|$. By Lemma \ref{puncturingbound},
the triangulations $P(\hat M, \hat T_1,x)$ and $P(\hat M, T^N,x)$
differ by a sequence of at most $100N$ interior Pachner moves,
up to ambient isotopy fixed on $\partial \hat M \cup \partial N(x)$.
By Lemma \ref{basepointchange},
$P(\hat M, T^N,x)$ and $P(\hat M, T^N, h^{-1}(y))$ 
differ by a sequence of at most $100 t_2$ Pachner moves, followed by a
simplicial isomorphism that acts as the identity on
$\partial \hat M$. Now, $P(\hat M, T^N, h^{-1}(y))$ and $P(\hat M, \hat T_2, y)$ are
homeomorphism-equivalent, via a homeomorphism that is the identity on $\partial \hat M$. 
The composition of the above homeomorphisms
restricts to a simplicial isomorphism $\phi \colon \partial N(x) \rightarrow
\partial N(y)$. By Lemma \ref{basepointchange},
$T_1 = P(\hat M, \hat T_1, y)$ and $P(\hat M, \hat T_1, x)$ differ by a sequence of
at most $100t_1$ interior Pachner moves, followed by a
simplicial isomorphism that acts as the identity on
$\partial \hat M$. Moreover, we may ensure that the
induced map $\partial N(y) \rightarrow \partial N(x)$ equals
$\phi^{-1}$. Thus, we have related $T_1$ and $T_2$
by a sequence of at most $100 P(t_1,t_2) + 100 (t_1 + t_2)$
interior Pachner moves, followed by a homemorphism that acts
as the identity on the boundary of $M$.
\end{proof}

\begin{proof}[Proof of Theorem \ref{3balltriangulations}]
We are given triangulations $T_1$ and $T_2$ of the 3-ball, $B$.
The links $L_1$ and $L_2$ and their regular neighbourhoods
$N(L_1)$ and $N(L_2)$ are subcomplexes. Let $M = B - {\rm int}(N(L_1))$. 
Then $T_1$ restricts to a triangulation $T'_1$ for $M$. Now,
$L_1$ and $L_2$ are assumed to be equivalent, and so there is a
homeomorphism $h$ of $B$ taking $N(L_1)$ to $N(L_2)$. After an
isotopy, we may assume that the restriction of $h$ to $N(L_1)$
realises the given combinatorial isomorphism between $N(L_1)$ and $N(L_2)$.
Let $T'_2$ be the triangulation of $M$ obtained by transferring
the restriction of $T_2$ to $M$ via $h^{-1}$. Thus, we have two triangulations
of $M$, $T'_1$ and $T'_2$, and these restrict to equal triangulations
of $\partial M$. Note that the number of tetrahedra in
$T'_1$ and $T'_2$ is at most $t_1$ and $t_2$ respectively.

Let $\hat M$ be the result of attaching a 3-ball to the 2-sphere boundary
component of $M$. By assumption, $P \colon \mathbb{N} \times \mathbb{N} \rightarrow
 \mathbb{N}$ is a Pachner move function for $\hat M$. Applying Theorem \ref{2sphereboundary},
we obtain a sequence of at most $100 P(t_1,t_2) + 100 (t_1 + t_2)$
interior Pachner moves, followed by a homeomorphism that
is the identity on $\partial M$, taking $T'_1$ to $T'_2$.
This induces a sequence of at most this many interior Pachner moves, followed by
a combinatorial isomorphism that is the identity on $\partial B$, taking $T_1$ to $T_2$ and $N(L_1)$
to $N(L_2)$. This completes the proof of Theorem \ref{3balltriangulations}.
\end{proof}

\section{Constructing a triangulation from a link diagram}\label{sec:triangs3}

In this section, we show how a link diagram can be used to construct a triangulation of the link's exterior.

\begin{define} The \emph{standard triangulation} of a cube is
obtained as follows. Start by inserting a vertex into each face, and coning off each
face from this vertex. This gives a triangulation of the boundary of
the cube. Now place a vertex at the centre of the
cube, and cone off the triangulation of the boundary. See Figure \ref{fig:stancube}. 
\end{define}

\begin{figure}[h!]
\centering
\includegraphics{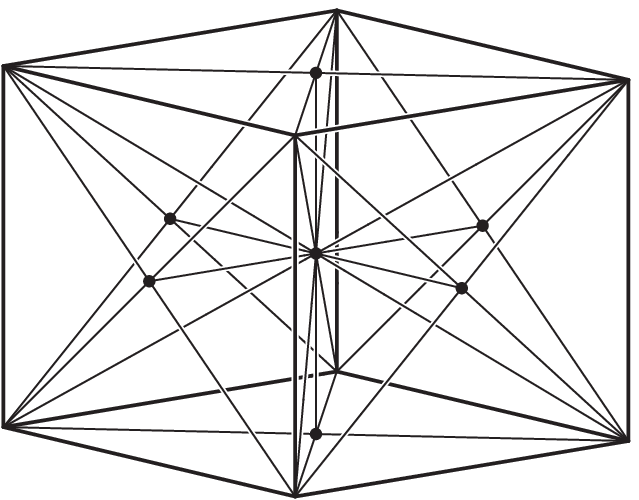}
\caption{} \label{fig:stancube}
\end{figure}

\begin{define} For any positive integer $n$, the \emph{standard triangulation
of the solid torus with length $n$} is obtained by gluing together $n$ cubes
in a circular fashion, with their standard triangulations, where the intersection
of a cube with its neighbours is precisely a pair of opposite faces of the
cube and where the other faces patch together to form four annuli. 
\end{define}

\begin{theorem} \label{trifromdiagram}
Let $L$ be a link in $S^3$ with components $L_1, \dots, L_r$. Let $D$ be a connected diagram of $L$ with $c(D)$ crossings. For $1 \leq i \leq r$,
let $c_i$ be the number of crossings in which at least one strand is part of $L_i$.
Let $n_1, \dots, n_r$ be integers satisfying $n_i \geq 10 c_i$ for each $i$,
and let $n = n_1 + \dots + n_r$.
Then there is a triangulation $T$ of a convex 3-ball in $\mathbb{R}^3$ with the following properties:
\begin{itemize}
\item it has at most $2^{12} c(D) + 120 (n+11) $ tetrahedra;
\item it contains $L$ as a subcomplex, and also a neighbourhood
$N(L) = N(L_1) \cup \dots \cup N(L_r)$ of $L$; 
\item the vertical projection of this copy of $L$ onto the horizontal plane
is the diagram $D$;
\item for each $i$, $N(L_i)$ has the standard triangulation of the solid torus
with length $n_i$;
\item each simplex of $T$ is straight in the affine structure on ${\mathbb R}^3$;
\item a collar neighbourhood of the boundary of the ball is triangulated
as $T_{S \times I}$.
\end{itemize}
\end{theorem}

\begin{proof} \noindent {\bf Step 1.} First suppose $D$ is not the trivial unknot diagram. Let $G$ be the underlying 4-valent planar graph of $D$.
We start by embedding $G$ in a square $Q$ in $\mathbb{R}^2$ so that each edge of $G$ is a union of at most 3 straight arcs. To find such an embedding, we first collapse parallel edges of $G$ to a single edge and remove edge loops, forming a graph $\overline G$. Using F\'ary's Theorem we may find an embedding of $\overline G$ in the plane in which every edge is straight. Now reinstate the parallel edges of $G$ with 2 straight arcs each and the edge loops of $G$ with 3 straight arcs. If $D$ is the trivial unknot diagram then let $G$ be a triangle formed of three straight edges in the interior of $Q$.

\vskip 6pt

\noindent {\bf Step 2.} Replace each edge of $G$ by 4 parallel edges.
Replace each 2-valent vertex of $G$ by 4 vertices joined by 3
straight edges. Replace each 4-valent vertex of $G$ by 9 parallelograms. See Figure \ref{fig:para9}. Call the resulting graph $G_+$.

\begin{figure}[h]
\centering
\psfrag{a}[Bl][Bl]{$\alpha$}
\includegraphics{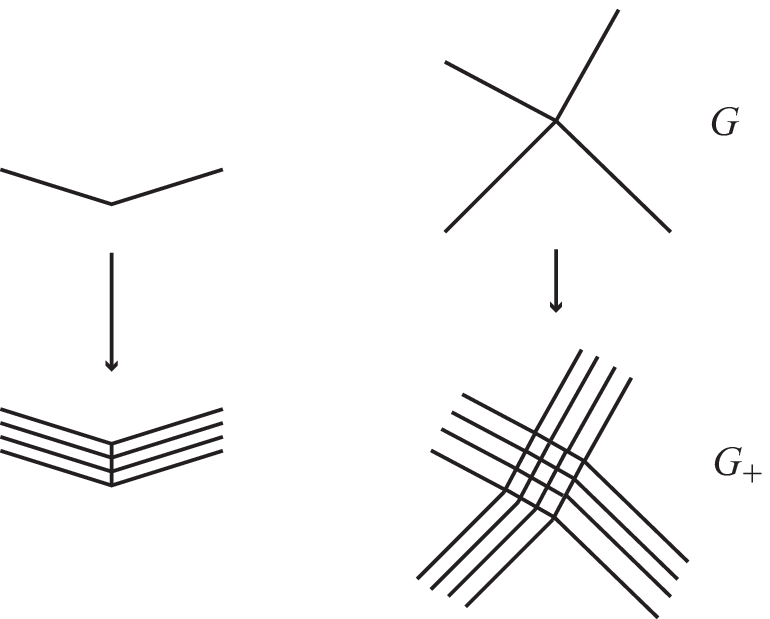}
\caption{} \label{fig:para9}
\end{figure}

\vskip 6pt

\noindent {\bf Step 3.} We will use $G_+$ to triangulate $Q$ as follows. Into each complementary region of $G_+$ coming from the complementary regions of $G$ we add straight edges until the region is triangulated. Denote by $E$ the union of these edges together with $G_+$ and $\partial Q$. An elementary Euler characteristic argument shows that $E$ decomposes $Q$ into at most $\textrm{max}(10c(D)+2,8) \leq 10c(D)+8$ triangles and at most $\textrm{max}(27c(D),9) \leq 27c(D)+9$ convex quadrilaterals.  Into the quadrilateral shaped regions add a vertex and cone from this vertex.  The result is a triangulation of the square $Q$ which we denote by $T_G$.

\vskip 6pt

\noindent {\bf Step 4.} Insert 4 copies of $Q$ with its triangulation $T_G$
into the cube $Q \times I$, one being the
top face, one the bottom face, and two parallel copies between them.
Insert a copy of $E \times I$, lying vertically in the cube,
running from top to bottom. The union of these copies of $Q$ with $E \times I$
decomposes the cube into a collection of convex balls. Triangulate each of the vertical faces by inserting
a vertex into the centre of the face, and coning off. Insert a vertex
into each ball and cone off.

The result is a triangulation of the cube. It has at most $2364 c(D) + 984$ tetrahedra.
However, it does not yet have all the required properties.

\vskip 6pt

\noindent {\bf Step 5.} We have triangulated the cube $Q \times I$, but we
actually require a triangulation of a ball so that a collar neighbourhood of
its boundary has the standard triangulation $T_{S \times I}$. Place a
large copy of $T_{S \times I}$ around $Q \times I$. We need to
triangulate the space between them. Do this by adding a cone over
each face of $Q \times I$, the cone point being a vertex of $T_{S \times I}$.
Then triangulate the remaining space. This certainly adds at most $300 + 236 c(D)$ tetrahedra
to the triangulation. 

\vskip 6pt

\noindent {\bf Step 6.} Near each 4-valent vertex of $G$, there are
9 parallelograms which are complementary regions of $G_+$. A copy of these 9
parallelograms lies in each of the 4 copies of $Q$,
and between these lies a parallelepiped made out of 27 smaller parallelepipeds which together look like a
Rubik's cube. We remove the entirety of this parallelepiped, and replace it with
a fixed triangulation of the parallelepiped with the same boundary.
This triangulation contains two pairs of parallelepipeds, glued
end-to-end, which together form two thickened arcs that
realise the crossing of $L$. We call these parallelepipeds \emph{crossing parallelepipeds}. See Figure \ref{fig:crosstri}.
We triangulate each of these crossing parallelepipeds using the standard triangulation of a cube.
We triangulate the rest of the large parallelepipeds using less than
$300$ tetrahedra each,
with their boundaries remaining unchanged from the start of this step.

\begin{figure}[h]
\centering
\includegraphics[width=0.4\textwidth]{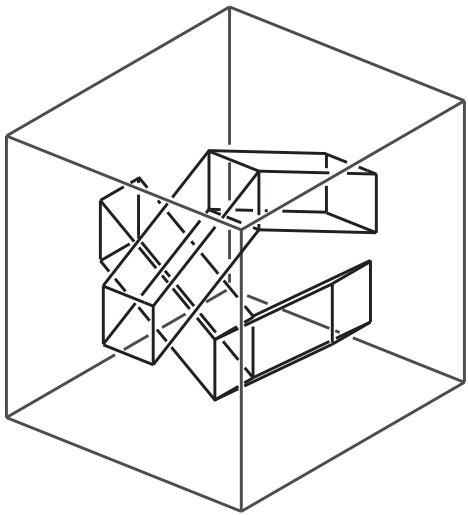}
\caption{Crossing parallelepipeds} \label{fig:crosstri}
\end{figure}

\vskip 6pt

\noindent {\bf Step 7.}
This new triangulation certainly has at most $2^{12}c(D) + 1284$ tetrahedra. It has all the required properties, with one
exception. Each component $N(L_i)$ of $N(L)$ is a union of cubes,
but the number of cubes is $n'_i$, say, where $n'_i \leq 10c_i$.

Thus, we remove each cube of $N(L_i)$, and replace it with the
following triangulation. Four of its outer faces have the
same triangulation as in the standard case. However, the faces $F$
that are attached to other cubes of $N(L_i)$ are given the
following triangulation. It has a central square, coned off.
There is an edge running from each vertex of this square to
the corresponding vertex of $F$. This creates 4 trapeziums.
A vertex is inserted into each, and then we cone off.
This specifies the boundary of each of these new cubes. See Figure \ref{fig:cubecube}.

\begin{figure}[h]
\centering
\psfrag{q}[Bc][Bl]{$\textrm{A portion of the boundary triangulation}$}
\psfrag{w}[Bc][Bl]{$\textrm{The interior cubes}$}

\includegraphics[width=0.8\textwidth]{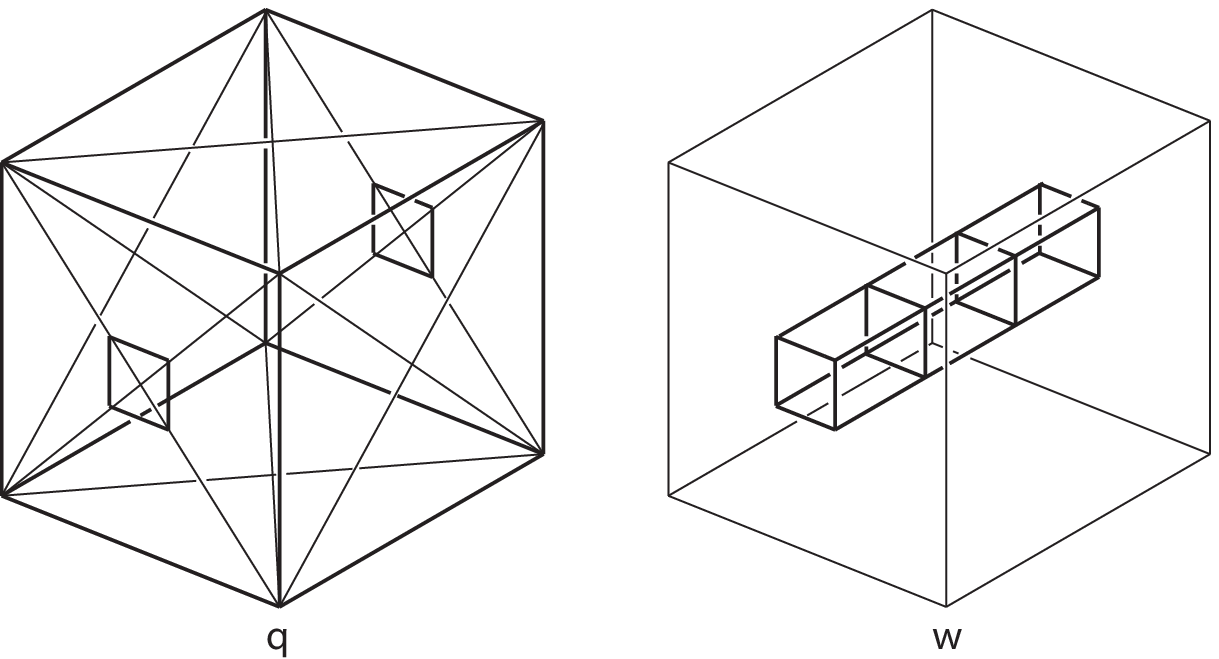}
\caption{} \label{fig:cubecube}
\end{figure}

Between opposite faces containing the new smaller squares,
we insert $\lfloor n_i / n'_i \rfloor$ or $\lceil n_i/n'_i \rceil$
standard cubes. The union of these cubes will form the new
$N(L)$. Between the boundary of the new $N(L)$ and the old
$N(L)$, we insert 2-simplices and 3-simplices so that the resulting triangulation has all the required properties. The number of tetrahedra lying within the old $N(L)$ is at most
$120 n$. \end{proof}

\pagebreak

\section{Alexander's trick and bounded isotopies} \label{sec:alexander}

In this section we prove the following result.

\begin{theorem} \label{alexandertrick}
 Let $B$ be a convex 3-ball in
$\mathbb{R}^3$. Let $h \colon B \rightarrow B$ be a
homeomorphism that fixes $\partial B$ pointwise and sends each straight arc in $B$
to a concatenation of at most $m$ straight arcs. Then there is an ambient isotopy $h_t \colon B \rightarrow B$ for
$t \in [0,1]$ such that $h_0$ is the identity map on $B$, $h_1=h$ and $h_t$
sends each straight arc to a concatenation of at most $m+2$
straight arcs in $\mathbb{R}^3$ for all $t \in [0,1]$.
\end{theorem}


\begin{proof} We will use Alexander's trick. Let $p$ be a point in the
interior of $B$. For each $r\in [0,1]$, let $B_r$ denote the result of
linearly scaling $B$ with centre $p$ by a factor of $r$, so that $B_1=B$ and
$B_0=\{ p \}$. More precisely, for $r>0$ let $B_r$ be the image of $B$
under the map $s_r:x \mapsto r(x-p)+p$.


Let $g_r:B_r \rightarrow B_r$ be given by $$g_r(x) =  s_r \circ h\circ
s_r^{-1}(x)$$ for $r>0$. The map $g_r$ essentially applies the homeomorphism
$h$ to the dilated copy of $B$, namely $B_r$, as shown in Figure
\ref{atrickgr}.

\begin{figure}[h]
\centering
\includegraphics{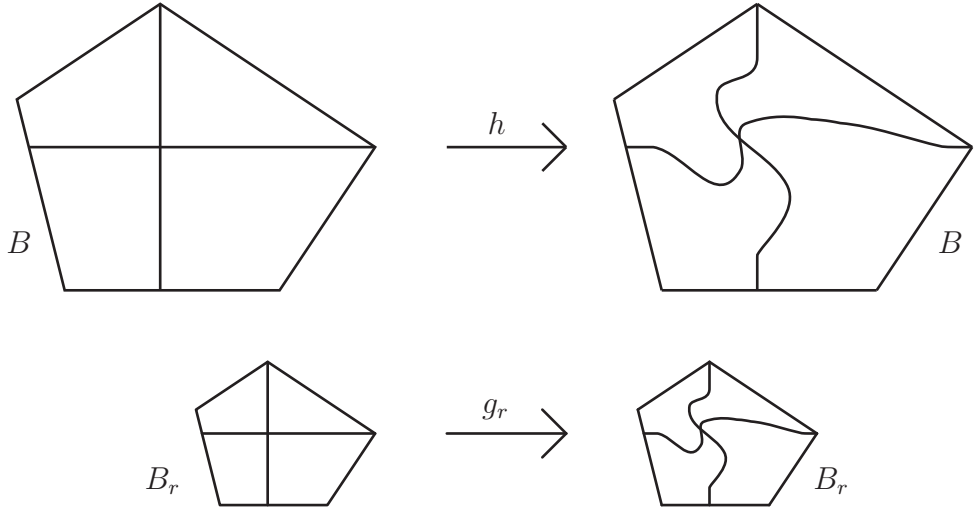}
\caption{The map $g_r$} \label{atrickgr}
\end{figure}

Alexander's trick may now be described as follows. Start with the identity
map on $B$ and set this to be $h_0$. As $t$ increases $h_t$ acts by applying
$g_t$ to $B_t$ and leaving the rest of $B$ unchanged, as shown in Figure
\ref{atrick}. When $t$ reaches 1, $h_1$ is the same as $h$ because $B_1 = B$
and $g_1 =h$.

\begin{figure}[h]
\centering
\includegraphics[width=100mm]{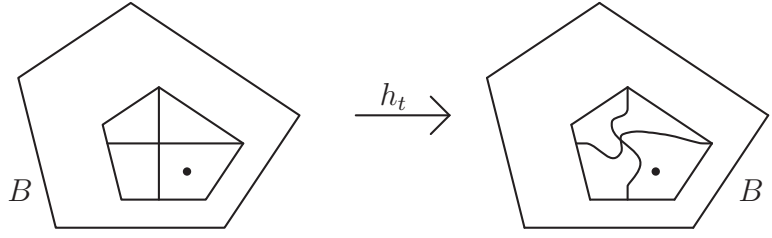}
\caption{The map $h_t$} \label{atrick}
\end{figure}

Formally, $h_t$ is given by

\begin{displaymath}
h_t(x) = \left \{ \begin{array}{ll}
g_t(x) & x \in B_t\\
x & x \notin B_t
\end{array}  \right. .
\end{displaymath}

Now, a straight arc in $B$ will be sent under $h_t$
to a single straight arc if it lies entirely outside $B_t$. It will be sent
to at most $m$ straight arcs if it lies entirely within $B_t$. If it has one
or both endpoints outside of $B_t$ but part of its interior within $B_t$
then it will be sent to at most $m+1$ or $m+2$ straight arcs respectively.
\end{proof}

\section{Continuous families of link projections}

In order to prove Theorem \ref{maintheorem}, we will want to find a sequence of
diagrams interpolating between two given diagrams for a link. However, in the course of the proof,
we will obtain not a sequence of diagrams, but a continuous family of
link projections. In this section, we will show how such a family can be used
to produce a sequence of diagrams.

We say that a piecewise linear map from a disjoint union of circles $C$ to the plane
is a \emph{link projection} if it can be factorised as $C \rightarrow \mathbb{R}^3 \rightarrow \mathbb{R}^2$,
where $C \rightarrow \mathbb{R}^3$ is an embedding and $\mathbb{R}^3 \rightarrow \mathbb{R}^2$
is the standard vertical projection onto the first two co-ordinates. In a link projection,
one keeps track not just of the map $C \rightarrow \mathbb{R}^2$, also one records,
for any two points in $C$ with the same image in $\mathbb{R}^2$, their relative heights
in $\mathbb{R}^3$. Link projections need not form diagrams in the usual sense.
For example, more than two points in $C$ may map to the same point in $\mathbb{R}^2$;
indeed uncountably many points may have the same image. However,
a link projection induces a \emph{diagram} if only finitely many points
in the plane have more than one inverse image point in the circles; each such point has
precisely two inverse image points; and near each such point in the plane, the
image of the link consists of two arcs intersecting transversely.

\begin{define}Let $D$ be a link diagram. Suppose that $D'$ is obtained from $D$ by adding a small unknot summand at a point in the interior of an edge of $D$. Then we say that $D'$ is obtained from $D$ by \emph{adding an unknot summand}. We say that  $D$ is obtained from $D'$ by \emph{removing an unknot summand}. 
\end{define}

\begin{proposition} \label{ctsfamily}
 Let $H \colon B \times [0,1] \rightarrow B$ be a piecewise
linear isotopy of a convex polyhedral 3-ball $B$ in $\mathbb{R}^3$. For each $t \in [0,1]$, let $h_t \colon B \rightarrow B$ be
$H(\cdot ,t)$. Suppose that $h_0$ is the identity.
Let $L$ be a piecewise linear link in the interior of $B$.
Suppose that, for each $t \in [0,1]$, $h_t(L)$ consists of
at most $n$ straight arcs. Suppose also that the projections of $h_0(L)$ and 
$h_1(L)$ are diagrams. Then, there is a sequence of diagrams
relating the projections of $h_0(L)$ and $h_1(L)$ with the following properties:
\begin{itemize}
\item successive diagrams are related by either a single Reidemeister move or the addition or removal of an unknot summand;
\item each diagram in the sequence has at most $n^2$ crossings.
\end{itemize}
\end{proposition}


Proving Proposition \ref{ctsfamily} is a fairly routine exercise in general position. Because of this, we only give an outline. 

Since $H$ is piecewise linear, $B \times [0,1]$ has a triangulation, $T$, such that $H$ restricts to a linear map on every simplex of $T$. Now, $L \times [0,1]$ intersects  every simplex of $T$ in a collection of affine pieces which we collectively call $A$. These affine pieces form a cell structure on the annulus $L \times [0,1]$. Let $t_1,t_2\ldots$ be the $t$-coordinates of the $0$-cells of $A$, arranged in increasing order. For every $i$, if $t \in (t_i, t_{i+1})$ then $L \times \{t\}$, consists of a fixed number of straight pieces, this number depending only on $i$. 

Let $D_t$ be the projection of the link $h_t(L)$ to the plane.
By the discussion above, on each interval $(t_i, t_{i+1})$, the projections $D_t$ vary in a fashion determined by the motion of the vertices of $H(L \times \{t\})$, with the edges in between remaining affine throughout. 

When $t=t_i$, the projections $D_t$ change in a way that is more complicated. Edges of $D_t$ can shrink to length zero as $t$ increases to $t_i$, and new edges and vertices can appear where none existed before when $t$ increases from $t_i$. Thus it is at these values of $t$ that we may need to remove and then add unknot summands.

In order to complete the proof of Proposition \ref{ctsfamily} we need to arrange that $D_t$ is a diagram for every $t$ apart from finitely many values of $t \in (0,1)$, at each of which a single Reidemeister move takes place or an unknot summand is removed and then another added. This we achieve by perturbing $H$.

First perturb $H$ so that each 0-cell in the interior of $A$ has a different
$t$-coordinate. Then $0 = t_0, t_1, \ldots , t_N = 1$ are the $t$-coordinates of the
0-cells of $A$. Now introduce two new vertices to each 1-cell in the interior of  $A$. If the
endpoints of the 1-cell have times $t_i$ and $t_j$, where $t_i < t_j$, then
place the vertices at $t_i + \varepsilon$ and $t_j - \varepsilon$, for some small
$\varepsilon > 0$. From now on, all perturbations of $H$ will be achieved by
slightly moving the location in $B$ of the images of the 0-cells of $A$.


The remainder of the proof of Proposition \ref{ctsfamily} will be in several steps. The first step will be to ensure that each $D_{t_i}$ is a diagram. Next we will ensure that for each $i$, $D_{t_i\pm \varepsilon}$ is, for  small enough  $\varepsilon$, a diagram related to $D_{t_i}$ by the addition or removal of an unknot summand. The third and final step will be to arrange  that the link projections between $D_{t_i + \varepsilon}$ and $D_{t_{i+1} - \varepsilon}$ are diagrams for all but finitely many times, $t$, at which a single Reidemeister move takes place. 

\textbf{Step 1 - Ensuring each $D_{t_i}$ is a diagram:} We need to perturb $H$ to ensure the following:
\begin{itemize}
\item any two edges of the projection $D_{t_i}$ intersect in at most one point;
\item no three edges of $D_{t_i}$ have a common point of intersection;
\item no vertex of the link has image that lies in a non-adjacent edge of $D_{t_i}$.
\end{itemize}
This is a straightforward general position argument. Suppose that two non-adjacent
edges intersect in more than one point. The images of their four endpoints
in $B$ lie in a 12-dimensional vector space. The subspace
consisting of configurations where the two lines intersect in more than one point
has dimension 10, for the following reason. The first line is specified by 6 co-ordinates.
The remaining line has projection that overlaps with the projection of
the first line, and so there are 4 co-ordinates that specify this line.
Since 10 is less than 12, we may perturb $h_{t_i}$ to avoid this subspace. Perturb $H$ to realize this perturbation of $h_{t_i}$. We can make this perturbation sufficiently small so that  no other bad configurations arise.
We can deal with the other cases similarly, and thereby ensure that
each $D_{t_i}$ is a link diagram. Note that the case where an edge of the link is projected to a single point is ruled out by the third condition.

\textbf{Step 2 - Perturbing each $D_{t_i\pm \varepsilon}$:} As $t$ increases from $t_i \in (0,1)$, the diagram  $D_{t_i}$ changes in the following way. All but one of the vertices of $D_{t_i}$ move with constant velocity. From the remaining vertex, $v$ say, a collection of vertices are formed and these all move away from $v$ with constant velocity. The velocity of these new vertices are the projected velocities of new vertices in $h_t(L)$ moving away from the pre-image of $v$. Because $D_{t_i}$ is a diagram, for $\varepsilon$ small enough, $D_{t_i + \varepsilon}$ is a diagram apart from possibly near $v$. To ensure that $D_{t_i + \varepsilon}$ is a diagram perturb it in a similar fashion to Step 1. All the diagrams $D_{t_i + \delta}$ change, as $\delta$ varies in $(0,\varepsilon]$, in a linear fashion with their vertices moving with constant velocity. Thus as $t$ increases from $t_i$ we have arranged that near $v$ we see a small unknot summand appear and grow in a linear fashion. A similar argument applies at times immediately prior to each $t_i \in (0,1)$. At times $t = 0$ and $t = 1$ things are only slightly different, with possibly several unknot summands appearing immediately after $t=0$ and possibly several unknot summands disappearing as $t$ reaches $1$.

\textbf{Step 3 - Passing from $D_{t_i + \varepsilon}$ to $D_{t_{(i+1)} - \varepsilon}$ with Reidemeister moves:} So far we have ensured that we may pass from the projection of $h_0(L)$ to the
projection of $h_1(L)$ by means of a sequence of diagrams where consecutive diagrams in the sequence are related by the addition or removal of a small unknot summand, or by moving their vertices with constant velocity, keeping the joining edges straight throughout. We ensure that the latter 1-parameter families of link projections may be perturbed to give rise to sequences of Reidemeister moves with the following proposition:

%
%
%

%
%
%
%
%
%
%
%

\begin{lemma} \label{lem:generalposition}
We may perturb each $D_{t_i \pm \varepsilon}$ slightly so that they remain diagrams,
and so that $D_t$ fails to be a diagram
for only finitely many values of $t$ in each  $[t_i+\varepsilon,t_{i+1}-\varepsilon]$. Moreover, at these values of $t$, a single Reidemeister
move is performed.
\end{lemma}

\begin{proof} 

We ensure that, for $t$ in $[t_i+\varepsilon,t_{i+1}-\varepsilon]$, the following conditions hold for every link projection $D_t$:
\begin{itemize}
\item for any two non-adjacent edges of the link, their projections intersect in at
most one point;
\item for any four edges of the link, their projections have no common point
of intersection;
\item the link projection has at most one triple point;
\item there is at most one vertex of the link that projects to a point
in the image of a non-incident edge;
\item no two vertices have a common projection.
\end{itemize}

We will also ensure that the following possibilities arise for at most finitely
many link projections:
\begin{enumerate}
\item the projection of some edge lies within the projection of an adjacent edge;
\item a vertex of the link has image lying within the projection of a non-incident edge;
\item there is a triple point.
\end{enumerate}

Note that there will, in general, be moments when (1), (2) or (3) above do occur, 
and typically then a Reidemeister move will be performed, as illustrated in Figure \ref{moves}.

\begin{figure}[h]
\centering
\includegraphics{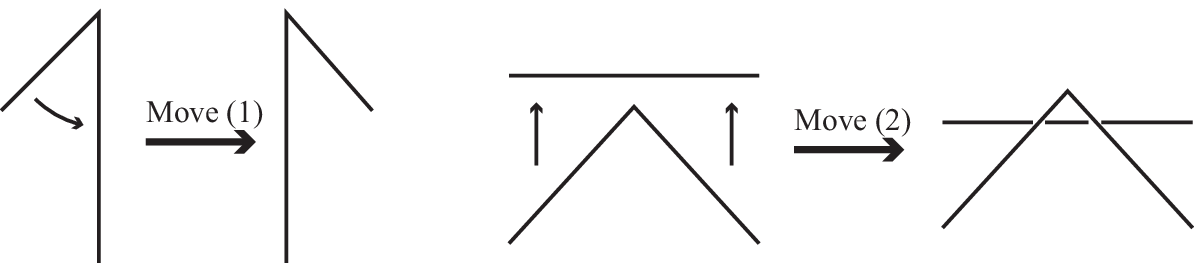}
\caption{} \label{moves}
\end{figure}

We start by arranging the first of these conditions. We will perturb the image of $L$ under the homeomorphisms
$h_{t_i + \varepsilon}$ and $h_{t_{i+1} - \varepsilon}$, but keeping $t_i+ \varepsilon$ and $t_{i+1}- \varepsilon$ fixed. 
Consider any two non-adjacent edges of the link. The images of their four endpoints under
$h_{t_i + \varepsilon}$ and $h_{t_{i+1} - \varepsilon}$ lie in a 24-dimensional vector space.
We wish to consider the subset of this space consisting of configurations
where the projections of the two edges intersect in more than one point. This subset lies in a subspace, the dimension of
which can be computed as follows. The two edges have bad projections at
some point in time. There is a one-dimensional family of possible times.
One of the edges has arbitrary image at this time, and so its position is specified
by 6 co-ordinates. The remaining edge must have projection lying in
the same line as the first edge. Thus, there are two parameters specifying
the endpoints of its projection. Two further parameters give the height
of these endpoints in ${\mathbb R}^3$. As one moves away from this time,
the endpoints of the edges vary linearly. Their derivatives are specified
by 12 further parameters. Thus, the subspace that we need to avoid has at most 23 dimensions.
Hence, by a small perturbation of $h_{t_i + \varepsilon}$ and $h_{t_{i+1} - \varepsilon}$, we can
avoid this subspace. Thus, we can ensure that, throughout the isotopy,
the projection of any two non-adjacent edges of the link intersect in at most one point.

We next ensure that,
throughout the isotopy, the projection of no four edges have a common
point of intersection. This time the ambient space, giving the
position of these four edges at times $t_i+ \varepsilon$ and $t_{i+1}- \varepsilon$ is
48-dimensional. The subspace containing the bad configurations has 47 dimensions:
one giving the time of the bad configuration, two parameters specifying the
point of intersection in the projection, four parameters give the angles
of the lines emanating from this point, eight further parameters give
the lines themselves, eight parameters give the heights in ${\mathbb R}^3$
of their endpoints, and 24 co-ordinates give the derivatives of
these endpoints at that time.

Thus, we can ensure that only triple points arise. Using similar arguments
to the ones above, we can ensure the remaining conditions. \end{proof}

%
%

This proves Proposition \ref{ctsfamily}.

\color{black}

Of course, Proposition \ref{ctsfamily} is not sufficient for our purposes because the addition or removal of unknot summands do not obviously induce a sequence of Reidemeister moves. For this reason we will need to apply the following theorem, whose proof, which we omit, is an easy adaptation of the methods used by Hass and Lagarias to prove the main theorem of \cite{HL}.

\begin{theorem}\label{hlmodified}
Let $D$ and $D'$ be link diagrams and suppose that $D'$ is obtained from $D$ by the addition or removal of an unknot summand. Suppose that $D$ and $D'$ both have at most $n$ crossings. Then there is a sequence of at most $2^{{10}^{11}n}$ Reidemeister moves relating $D$ and $D'$.
\end{theorem}

Thus, from a continuous family of link projections interpolating
between $D_1$ and $D_2$, we can find a sequence of Reidemeister moves taking
$D_1$ to $D_2$. In order to efficiently bound the number of moves required we will need the following theorem.

\begin{theorem} \label{numberofdiagrams}
Up to ambient isotopy, there are at most $(24)^{n+1}$ connected, unoriented link diagrams with at most
$n$ crossings.
\end{theorem}

\begin{proof} This is an adaptation of an argument of Welsh \cite{W}. We use a theorem of Tutte from \cite{tutte}, which counts `rooted bicubic maps'.
A \emph{map} is a cell structure on the 2-sphere, or equivalently an embedded connected planar
graph. It is \emph{trivalent} if the valence of each of its vertices is $3$.
It is \emph{bicubic} if it is trivalent and bipartite. A map is
\emph{rooted} if one of its edges is chosen and oriented, and the two faces on
either side of this edge are specified as lying on the left and right of the edge.
The point of using rooted maps is that they have no symmetries, and so they
are easier to count. It is easy to show that the faces of a bicubic map
may be coloured using only three colours, and so that adjacent faces have distinct
colours. This colouring is unique once the colours adjacent to some edge are chosen.
If the bicubic map is rooted, then the \emph{root colour} is the face colour not
adjacent to the specified edge. Tutte proved in 4.3 of \cite{tutte} that the number of
rooted bicubic maps in which there are just $n$ faces of the root colour is
$${2(2n)!3^n \over n! (n+2)!}.$$
Now, given a connected link diagram, there is a simple way of creating a bicubic map.
One simply replaces each 4-valent vertex of the link projection with a square consisting
of four vertices and four edges. One can root the bicubic map by picking one of the
original edges of the diagram. The new squares are then assigned the root colour.
Hence, we deduce that the number of connected rooted embedded 4-valent planar graphs with $n$ vertices is
also at most 
$${2(2n)!3^n \over n! (n+2)!}.$$
There are $2^n$ ways of assigning the crossing information and so there are at most
$${2(2n)!6^n \over n! (n+2)!}$$
connected link diagrams with $n$ crossings. So, the number with at most $n$ crossings is no more than
$$\sum_{k=0}^n {2(2k)!6^k \over k! (k+2)!} \leq \sum_{k=0}^n 2 \left ( {2k \atop k} \right ) {6^k \over (k+1)^2}
\leq \left ( \sum_{k=0}^n 24^k \right ) \left (\sum_{k=1}^\infty {1 \over k^2} \right ) \leq 24^{n+1},$$
as required.
\end{proof}

Theorem  \ref{numberofdiagrams} yields the following immediate corollary:

\begin{cor} \label{numberofdiagramscor}Up to ambient isotopy, there are at most $(48)^{n+1}$
connected, oriented link diagrams with at most n crossings.
\end{cor}

We may combine Theorem \ref{hlmodified} and Corollary \ref{numberofdiagramscor} to obtain the following corollary.

\begin{cor} \label{numberofmoves}
Suppose that there is a finite sequence of diagrams,
starting with $D_1$ and ending with $D_2$, so that successive diagrams are
related by either a Reidemeister move or the addition or removal of an unknot summand. Suppose that each diagram in this sequence is connected
and has at most $n$ crossings. Then, there is a sequence of at most  $2^{{10}^{12}n}$ Reidemeister
moves joining $D_1$ to $D_2$.
\end{cor}

\begin{proof}
Without loss of generality suppose that the hypothesized sequence of diagrams is chosen to be as short as possible. Then no two diagrams are isotopic and so, by Corollary \ref{numberofdiagramscor}, the sequence consists of at most $(48)^{n+1}$ diagrams. By Theorem \ref{hlmodified}, passing from one diagram to the next can be achieved with at most $2^{{10}^{11}n}$ Reidemeister moves. The product of these two expressions is less than the desired bound.
\end{proof}

%

%
%
%
%
%
%
%

Combining Theorem \ref{alexandertrick}, Proposition \ref{ctsfamily} and Corollary \ref{numberofmoves}, we obtain the
following.

\begin{corollary} \label{reidbound}
Let $h \colon B \rightarrow B$ be a PL homeomorphism of a convex polyhedral 3-ball $B$, which fixes  $\partial B$ pointwise, and which has the
property that it sends each straight arc in $B$ to a concatenation of at most $m$
straight arcs. Let $L$ be an non-split oriented link in $B$ that is the concatenation of at most $n$
straight arcs. Suppose that $L$ and $h(L)$ project to oriented diagrams $D_1$ and $D_2$.
Then $D_1$ and $D_2$ differ by a sequence of at most
$$2^{{10}^{12}(n(m+2))^2} \leq 2^{{10}^{13}(nm)^2}$$
Reidemeister moves.
\end{corollary}

\section{Proof of the main theorem}

In this section, we will prove Theorem \ref{fromptor}. This will quickly yield Theorem \ref{maintheorem}, assuming Theorem \ref{mijatovicmodified}. All diagrams in this section will be oriented. Note that Theorem  \ref{maintheorem} for unoriented diagrams follows from the version for oriented diagrams. 

Let $D_1$ and $D_2$ be connected diagrams for some knot or non-split link $L$ in the 3-sphere.
Let $n_1$ and $n_2$ be their crossing numbers, and let $n = n_1 + n_2$.
We wish to find a sequence of diagrams taking $D_1$ to $D_2$, where successive
diagrams are related by a Reidemeister move.

We suppose that $L$ is not the unknot, for in this case Theorem \ref{fromptor} follows from \cite{HL}. (Theorem  \ref{maintheorem} also follows immediately from  \cite{HL} in this case.) 

Start by applying type 1 Reidemeister moves to $D_1$ and $D_2$ so the writhes of corresponding components agree. Call the resulting diagrams $D_1'$ and $D_2'$ respectively. They have at most $2n_1$ and $2n_2$ crossings respectively. The number of Reidemeister moves required is at most $n$.

Use Theorem \ref{trifromdiagram} to create triangulations $T_1$ and
$T_2$ of the 3-ball $B$ with $L_1$ and $L_2$ as subcomplexes.
The vertical projections of $L_1$ and $L_2$ are the diagrams $D_1'$ and $D_2'$. Theorem  \ref{trifromdiagram}  ensures
that the triangulations $T_1$ and $T_2$ contain neighbourhoods $N(L_1)$ and $N(L_2)$ of $L_1$ and $L_2$ that
are subcomplexes and that each have a standard triangulation
of the solid torus. By choosing the integers $n_i$ in Theorem
 \ref{trifromdiagram}  appropriately, we may arrange that these triangulations of the solid
tori have the same length, and hence are combinatorially
isomorphic. Further, since  $D_1'$ and $D_2'$ have the same writhes, we may also ensure that the homeomorphism
$N(L_1) \rightarrow N(L_2)$ that realises this combinatorial
isomorphism preserves longitudes. Since $2^{13}n_1+120(40(n_1+n_2)+11) < 2^{14}n$, the triangulations $T_1$ and $T_2$  may be taken to contain at most $t = 2^{14}n$ tetrahedra each.
Apply Theorem \ref{3balltriangulations} to give a sequence of at most
$N = 100 P_L(t,t) + 200t$ interior Pachner moves, followed by a
combinatorial isomorphism, that takes $T_1$ to $T_2$, and $N(L_1)$
to $N(L_2)$. Further, none of the Pachner moves affect $N(L_1)$ or $\partial B$, and the combinatorial isomorphism restricts to the identity on $\partial B$. This sequence of Pachner moves gives a sequence of
$\Delta$-complexes joined by PL homeomorphisms
$$|T_1| = |T^0| \buildrel h_1 \over \longrightarrow |T^1|
\buildrel h_1 \over \longrightarrow \dots \buildrel h_{N} \over \longrightarrow |T^N| = |T_2|.$$
Let $h = \phi_2^{-1} h_{N} \dots h_1 \phi_1$ be the resulting homeomorphism
$B \rightarrow B$, where $\phi_1$ and $\phi_2$ are the homeomorphisms associated with $T_1$ and $T_2$. Note that $h(L_1) = L_2$.

We will need the following concept. Let $T$ be a triangulation of a
3-manifold. Then an arc in $|T|$ is said to be \emph{straight} if it lies in a
single simplex and is straight in the affine structure on that
simplex. We now apply the following straightforward lemma.

\begin{lemma} \label{straight2}
Suppose that $T^i$ and $T^{i+1}$ are triangulations of a 3-manifold that
are related by an interior Pachner move. Let $h_{i+1} \colon |T^i| \rightarrow |T^{i+1}|$ be
the resulting homeomorphism. Then, $h_{i+1}$ sends each straight arc in $|T^i|$
to a concatenation of at most $4$ straight arcs in $|T^{i+1}|$.
\end{lemma}

Now, $\phi_1$ sends each straight arc in $B$ to a concatenation of at most $t$
straight arcs in $|T_1|$. By Lemma \ref{straight2}, this is sent to a concatenation of
at most $4^Nt$ straight arcs in $|T_2|$. Each of these arcs
is sent to a straight arc in $B$. Thus we obtain the following.

\begin{corollary} \label{straight1}
 Each straight arc in $B$ is sent, via $h$, to a concatenation of at most $4^Nt$
straight arcs.
\end{corollary}

Now, since $L_1$ lies in the 1-skeleton of $T_1$, it consists of at most $6t = 6 \cdot 2^{14}n$ straight arcs in $B$. Hence we may apply Corollary \ref{reidbound} to conclude that  there is a sequence of at most 
$$2^{{10}^{13}(6t\cdot 4^Nt)^2} \leq \exp({{10}^{15}t^42^{4N}})$$ Reidemeister moves taking $D_1$ to $D_2$. 

This proves Theorem  \ref{fromptor} since
\begin{align*}
\exp({{10}^{15}t^42^{4N}})  &= \exp({{10}^{15}t^42^{4 (100 P_L(t,t) + 200t)}}) \\
&\leq  \exp( 2^{50 +4t + 4 (100 P_L(t,t) + 200t)} )\\
&\leq  \exp( 2^{400(P_L(t,t) + 4t)} )\\
&= \exp^{\left( 2 \right) }({400(P_L(2^{14}n,2^{14}n) + 2^{16}n)} ).
\end{align*}

Now, Theorem \ref{mijatovicmodified} may be applied to non-split link exteriors because the 3-sphere cannot contain
a 3-dimensional submanifold that has a single boundary component, fibres over
the circle, and is a surface semi-bundle. Indeed, no ${\mathbb Z}_2$-homology 3-sphere
can contain such a submanifold. Hence Theorem \ref{maintheorem} follows
in the case where the link is not split, assuming Theorem \ref{mijatovicmodified}, because
$$ \exp^{\left( 2 \right) }({400(2\exp^{\left( a^{t} \right) }(t) + 4t)} ) \leq  \exp^{\left( 2 \right) }{(\exp^{\left( a^{t} + 10 \right) }(t))}  = \exp^{\left( a^{t} + 12 \right) }(t) \leq \exp^{(c^n)}(n),$$ where $c = 2^{163\cdot 2^{14}} \leq 10^{1,000,000}$. 

It remains to prove Theorem \ref{maintheorem} in the case where $L$ is split. In this case, we apply the following result.

\begin{proposition} \label{splitlink}
Let $L$ be a split link, and
write $L = L_1 \cup \dots \cup L_r$, where each $L_i$
is non-split, and when $i \not= j$, $L_i$ and $L_j$ are separated by a
sphere that lies in the complement of $L$. Let $D$ be a diagram
for $L$ with $n$ crossings, and let $D_i$ be the restriction of this diagram to
$L_i$. Then there is a sequence of at most $2(r-1)\exp^{\left(k^n \right)}(k^n)$
Reidemeister moves taking $D$ to the distant union of $D_1, \dots, D_r$,
where $k = 2^{129}$.
\end{proposition}

\begin{proof}
It is a theorem of Hayashi \cite{hayashi} that if $D$ is a diagram of a split link with $n$
crossings, then there is a sequence of at most $\exp^{\left(k^n \right)}(k^n)$
Reidemeister moves taking it to a disconnected diagram $D'$, where $k = 2^{129}$.
(We have simplified Hayashi's bound, at the cost of a slight increase to it.)
Suppose that this disconnected diagram separates the link into two
subsets $K_1$ and $K_2$. Then, the restriction of $D$ to each $K_i$
is a diagram $E_i$ for $K_i$. We may now reverse some of these Reidemeister
moves to take $D'$ to the distant union of $E_1$ and $E_2$.
We apply this argument $r-1$ times, and we end with the distant
union of $D_1, \dots, D_r$.
\end{proof}

Write $L = L_1 \cup \dots \cup L_r$, where each $L_i$
is non-split, and when $i \not= j$, $L_i$ and $L_j$ are separated by a
sphere that lies in the complement of $L$. From $D_1$, we obtain diagrams
$D_{1,1}, \dots, D_{1,r}$ for $L_1, \dots, L_r$, by disregarding components
not belonging to the relevant $L_i$. Let $D_1'$ be the distant union of the
diagrams $D_{1,1}, \dots, D_{1,r}$.
Similarly, we obtain diagrams $D_{2,1}, \dots, D_{2,r}$ from $D_2$, and let
$D_2'$ be their distant union. Let $n_{i,j}$
be the number of crossings in $D_{i,j}$. Thus, $n_{1,1} + \dots + n_{1,r} \leq n_1$,
and $n_{2,1} + \dots + n_{2,r} \leq n_2$. Let $m_j = n_{1,j} + n_{2,j}$.
By Theorem \ref{maintheorem} in the non-split case, there is a sequence of at most
$$\exp^{\left ( c^{m_1} \right ) }(m_1) + \dots + \exp^{\left ( c^{m_r} \right ) }(m_r)$$
Reidemeister moves taking $D_1'$ to $D_2'$. By Proposition \ref{splitlink},
there is a sequence of at most $2(r-1)\exp^{(k^{n_1})}(k^{n_1})$ Reidemeister moves taking $D_1$ to $D_1'$,
and similar number taking $D_2$ to $D_2'$. The total number of moves is less than
$\exp^{(c^n)}(n)$. This concludes the proof of Theorem \ref{maintheorem}, assuming Theorem  \ref{mijatovicmodified}.

\section{Overview of the proof of Mijatovi\'c's theorem (Simple case)}
\label{sec:mijatovic-simple}

The driving force behind this paper is Theorem \ref{mijatovicmodified}, which is based on work of Mijatovi\'c.
As mentioned in the Introduction, Mijatovi\'c in fact proved the following related result \cite{alexknot}.

\begin{theorem} \label{mijatovichaken}
Let $M$ be a compact orientable irreducible 3-manifold, with
boundary a non-empty collection of tori.
Suppose that the closure of each component of the complement of the characteristic
submanifold of $M$ satisfies at least one of the following conditions:
\begin{itemize}
\item it does not fibre over the circle; or
\item it is not a surface semi-bundle; or
\item it has at least two boundary components.
\end{itemize}
Let $T_1$ and $T_2$ be two triangulations of $M$ with $t_1$ and $t_2$
3-simplices respectively.
Then there is a sequence of at most
$\exp^{( b^{t_1})}(t_1) + \exp^{( b^{t_2} )}( t_2)$
interior Pachner moves and boundary
Pachner moves, followed by an ambient isotopy, followed by a homeomorphism of $M$
to itself that is supported in the characteristic submanifold, that leaves each
component of $\partial M$ invariant, and that takes $T_1$ to $T_2$.
Here, $b \leq 2^{200}$.
\end{theorem}

There are two main differences between the statements of Theorems
\ref{mijatovichaken} and \ref{mijatovicmodified}.
Firstly, boundary Pachner moves are permitted in Theorem \ref{mijatovichaken}
but not in Theorem \ref{mijatovicmodified}.
Secondly, the homeomorphism of $M$ is not required to be the identity on the
boundary of $M$ in Theorem \ref{mijatovichaken}.
We will explain how to guarantee these extra requirements in Sections \ref{sec:canonical-tri} - \ref{sec:adjusting-boundaries}, 
following a summary of Mijatovi\'c's techniques in this and the next section.

We briefly recall what is meant by the \emph{characteristic submanifold} of $M$. We first consider the case where $M$ is compact, orientable and irreducible, and has boundary a (possibly empty) collection of incompressible tori. A torus properly embedded in $M$ is \emph{canonical} if it is essential and, moreover,
any other essential torus can be isotoped off it. If one takes one representative
for each ambient isotopy class of canonical torus, then these can be chosen to be
disjoint. The resulting collection of tori is the \emph{JSJ tori}, and their
union is well-defined up to ambient isotopy. A key result in the theory is that
if one cuts $M$ along an open regular neighbourhood of the JSJ tori, the resulting
pieces are either Seifert fibred or simple. The union of the Seifert fibred pieces
is the \emph{characteristic submanifold} of $M$.

When $M$ has boundary that is incompressible, but not a union of tori, then
one must vary the above definition. Here, $M$ is still assumed to be compact, orientable and irreducible. One considers essential annuli and tori properly
embedded in $M$, and again such a surface is \emph{canonical} if any other essential annulus
or torus can be isotoped off it. To form the \emph{JSJ annuli and tori}, one takes one 
representative of each isotopy class of canonical annulus and torus, but one then
discards certain annuli, called \emph{matching annuli}, which have Seifert fibred
spaces on both sides with matching Seifert fibrations. In this case, the JSJ annuli and tori divide $M$ into pieces that
are simple, Seifert fibred or an $I$-bundle over a surface, and the
\emph{characteristic submanifold} is the union of the Seifert fibred and
$I$-bundle pieces. For further details on JSJ decompositions of 3-manifolds, see  \cite{JS} and \cite{JSJ2}. See also \cite{JSJ3}.

The idea behind Theorem \ref{mijatovichaken} is as follows. As in Theorem \ref{mijatovicmodified},
a `canonical triangulation' ${\mathcal T}_{\rm can}$ for $M$ is constructed. However, this is slightly
different from the triangulation ${\mathcal T}_{\rm can}'$ in Theorem \ref{mijatovicmodified}. 
Like ${\mathcal T}_{\rm can}'$,
${\mathcal T}_{\rm can}$ depends on a given triangulation of $\partial M$. However, its restriction
to $\partial M$ does not equal this triangulation.

Mijatovi\'c's triangulation ${\mathcal T}_{\rm can}$ is constructed from a collection of surfaces
in $M$. These include the JSJ tori of $M$. In each component of the complement of the
JSJ tori, the choice of surfaces depends on whether that piece is
simple or Seifert fibred. Therefore, in this section, we will focus on the
case where $M$ is simple. In the next section, we will examine the case where $M$
is Seifert fibred. In Section \ref{sec:canonical-tri}, we will consider the 
general case.

Suppose therefore that $M$ is a compact orientable simple Haken 3-manifold
satisfying the conditions of Theorem \ref{mijatovichaken}.
A \emph{partial hierarchy} for $M$ is a sequence
$$M = M_0 \buildrel S_1 \over \rightsquigarrow M_1 \buildrel S_2 \over
\rightsquigarrow \dots \buildrel S_n \over \rightsquigarrow M_n,$$
where each $M_i$ is a 3-dimensional submanifold of $M$, each $S_i$
is a properly embedded incompressible surface in $M_{i-1}$, and $M_i$
is obtained from $M_{i-1}$ by cutting along $S_i$. A \emph{hierarchy}
is a partial hierarchy where the final manifold is a collection
of 3-balls. At each stage, one keeps track of a \emph{boundary pattern} $P_i$ for $M_i$. 
Here, we are using this term to mean a collection of disjoint
simple closed curves and graphs embedded in $\partial M_i$.
This boundary pattern is defined as follows.

One considers the union of the surfaces $S_1, \dots, S_i$
as a 2-complex in $M$, so that each $S_j$ has boundary that
runs over $S_1 \cup \dots \cup S_{j-1}$. Each surface $S_j$ is required to be in general
position with respect to the previous surfaces, in the sense
that its boundary is required to be transverse to the union of
the boundaries of the earlier surfaces. In general, one also
wants to ensure that when two parts of $\partial S_j$ intersect because they
lie on different sides of some previous surface, then they also are
required to be in general position. However, this situation will never
in fact arise in this paper, because we will always cut along separating
surfaces. The boundary pattern $P_i$ for $M_i$
is defined to be the image of $\partial S_1 \cup \dots \cup \partial S_i$
in $\partial M_i$. This boundary pattern is \emph{essential} which means
that $\partial M_i - P_i$ is incompressible in $M_i$.

Mijatovi\'c's canonical triangulation ${\mathcal T}_{\rm can}$ is constructed as follows.
The union of $\partial M$ and the surfaces $S_1, \dots, S_n$
is a 2-complex. Triangulate each face of this complex, by inserting a vertex in its interior,
and then coning off. The boundary of each 3-ball in $M_n$ then inherits a triangulation.
Triangulate this 3-ball by coning off this boundary triangulation.
Provided one uses the right hierarchy, the result is ${\mathcal T}_{\rm can}$.

Not any choice of hierarchy will work here. It is important that
whenever $M_{i-1}$ is given some triangulation, then $S_i$ can be realised
as a normal surface in that triangulation with an estimable number of triangles and squares. It is also important that there is a bound on the length $n$
of the hierarchy, in terms of the number of tetrahedra in any
given triangulation of $M$.
The reason for this will shortly become apparent.

To prove Theorem \ref{mijatovichaken} or \ref{mijatovicmodified}, one starts with a triangulation $T$ for $M$.
Let's also call this triangulation $T_0(M)$ and $T_0(M_0)$.
One realises $S_1$ as a normal surface in this triangulation, with control over the number
of triangles and squares. Then one applies Pachner moves to $T_0(M)$, creating a triangulation $T_1(M)$ in which
$S_1$ is a subcomplex. This restricts to a triangulation $T_1(M_1)$ of $M_1$. Then one repeats,
by realising $S_2$ as a normal surface in $T_1(M_1)$, and so on. At each stage,
the boundary pattern $P_i$ is required to be a subcomplex of the
triangulation $T_i(M_i)$. The final result is a triangulation $T_n(M_n)$ of $M_n$, 
the components of which patch together to form a triangulation of $T_n(M)$ of $M$ in which
$S_1 \cup \dots \cup S_{n}$ is a subcomplex. Now one applies
Pachner moves to $T_n(M)$ so that the triangulation on the 2-complex $S_1 \cup \dots \cup S_{n}$
agrees with that of ${\mathcal T}_{\rm can}$. Thus, on each component of $M_n$, we have two triangulations of
the 3-ball which agree on their boundaries. Then Mijatovi\'c applies an earlier result 
(Theorem 5.2 of \cite{alexsf}), which allows one to bound the number of interior
Pachner moves required to pass between two such triangulations of a 3-ball.

The precise choice of hierarchy made by Mijatovi\'c is
rather delicate. The surfaces that he uses are designed to
ensure that one can give an upper bound on the number of triangles and
squares of the normal surface $S_i$ in any given
triangulation of $M_{i-1}$, and also to provide a bound
on the length of the hierarchy. 

The first surface $S_1$ must be chosen with particular care.
In order that later stages of the hierarchy can be chosen
canonically, it is important that the exterior of $S_1$ is not a union of
$I$-bundles. In other words, $S_1$ must be neither a
fibre in a fibration of $M$ over the circle nor a fibre in
a surface semi-bundle. The hypotheses on $M$ in Theorem \ref{mijatovichaken} are
there to ensure that it is possible find a properly embedded
essential surface satisfying this condition. For when $M$
is a compact orientable irreducible atoroidal 3-manifold with
boundary a single incompressible torus, then a theorem of
Culler and Shalen \cite{cullshal} gives that $M$ contains
a properly embedded connected essential surface that is either separating
or closed. Thus, it is not a fibre. It may be a
semi-fibre, but our hypotheses on $M$ then ensure
that $M$ does not fibre over the circle and so one may
instead use a non-separating surface. When $M$ has
more than one boundary component and is not homeomorphic to $S^1 \times S^1 \times I$, 
one may use Culler and Shalen's theorem to find an essential
properly embedded surface that intersects at most one component of
$\partial M$. Hence, in this case, it is neither a fibre
nor a semi-fibre. See Corollary 3.2 in \cite{alexknot} for
more details. It is possible to find such a surface in
normal form, with a bound on its number of triangles and squares using
Propositions 4.1 and 4.2 of \cite{alexknot}.

At each of the later stages of the hierarchy, one of the
following types of surface is used: \begin{enumerate}
\item the boundary of a regular neighbourhood of a closed,
connected, properly embedded, $\pi_1$-injective surface of maximal Euler characteristic  in
some component of $M_{i-1}$ minus its characteristic submanifold (such a
surface is only used in the initial stages when the boundary pattern is
empty);
\item a canonical annulus in some component of $M_{i-1}$,
and which is also disjoint from the boundary pattern $P_{i-1}$
(here, canonical is defined in terms of the boundary pattern $P_{i-1}$);
\item an incompressible annulus (or two parallel copies of
such an annulus) in a component of $M_{i-1}$
that is either an $I$-bundle over a surface or a compression body, which
joins different boundary components, and which has minimal
intersection number with $P_{i-1}$ (such a surface is
only used when the component of $M_{i-1}$, with its boundary pattern, has
no canonical annuli and empty characteristic submanifold);
\item a meridian disc (or perhaps two parallel copies of
such a disc) for a handlebody component of $M_{i-1}$,
and which has minimal intersection number with $P_{i-1}$
(again such a surface is
only used when the component of $M_{i-1}$, with its boundary pattern, has
no canonical annuli and empty characteristic submanifold).
\end{enumerate}
We refer the reader to \cite{alexff} for the precise order in which
these surfaces are used. There, it is also shown that
$S_i$ may be realised as a normal surface in $T_{i-1}(M_{i-1})$
with bounded complexity, and a bound on the length of
the hierarchy is also given.

We now explain how one passes from the triangulation
$T_{i-1}(M)$ to the triangulation $T_i(M)$. The surface $S_i$ is a normal surface in $T_{i-1}(M_{i-1})$.
The triangulation $T_i(M)$ is chosen so that $S_i$ is simplicial in it. 
More precisely, $S_i$ intersects each tetrahedron of $T_{i-1}(M)$ in a collection of
triangles and squares. A vertex is inserted into each of these triangles and squares, and then the
triangle or square is coned off. The surface $S_i$ divides each face of $T_{i-1}(M)$
into discs. A vertex is inserted into each of these discs, and this too is coned off. Now, each tetrahedron of $T_{i-1}(M)$ is divided into balls by $S_i$. A vertex is
placed in the interior of each of these, and then the ball is coned off.
The resulting triangulation of $M$ is $T_i(M)$. Its restriction to $M_i$
is $T_i(M_i)$. 

We now give Mijatovi\'c's sequence of Pachner moves
which takes $T_{i-1}(M)$ to $T_i(M)$. This takes place in 6 steps:
\begin{enumerate}
\item Perform a $(1,3)$ boundary Pachner move on each triangle in $\partial M$,
by attaching a tetrahedron to it. Then, in a similar fashion, make a $(2,2)$ boundary Pachner move
for each 1-simplex of $T_{i-1}(M)$ in $\partial M$.
\item Add a vertex into each tetrahedron of $T_{i-1}(M)$ (but not the newly attached
tetrahedra from Step 1) by performing a $(1,4)$ Pachner move. Then add a vertex to each
triangle of $T_{i-1}(M)$ by performing a $(1,4)$ move on an adjacent tetrahedron and then
a $(2,3)$ move.
\item Subdivide the 1-skeleton of $T_{i-1}(M)$ so that it becomes a subcomplex
of $T_i(M)$, and keep the triangulation of the 3-simplices of $T_{i-1}(M)$ coned.
The precise details of how to do this are in Step 2 in Section 5 of \cite{alexs3}.
\item Subdivide the 2-skeleton of $T_{i-1}(M)$ to get a subcomplex of $T_i(M)$,
and keep the triangulation of the 3-simplices of $T_{i-1}(M)$ coned. The process
here is described in Lemma 4.2 of \cite{alexs3}.
\item Chop up the tetrahedra of $T_{i-1}(M)$ along the normal triangles and
squares of $S_i$ and triangulate the complementary regions by coning them from
points in their interiors. One uses Lemma 5.1 of \cite{alexs3} to do this.
\item Finally, remove the tetrahedra that are not contained in any of the
3-simplices of $T_{i-1}(M)$ by performing boundary Pachner moves.
\end{enumerate}
We will not need all the details of this process. But the following observation will
be important for us. The boundary Pachner moves that are used depend only on
the intersection between $\partial S_i$ and $\partial M$.  In fact, the way they
arise is precisely as follows:
\begin{enumerate}
\item A $(1,3)$ move is performed on each triangle of $T_{i-1}(M) \cap \partial M$, which attaches
a tetrahedron.
\item A $(2,2)$ move is performed on each edge of $T_{i-1}(M) \cap \partial M$, which again attaches
a tetrahedron.
\item The newly introduced tetrahedra are then modified using interior Pachner moves.
However, these moves are determined entirely by $\partial S_i \cap \partial M$.
\item Boundary Pachner moves are then performed which remove the
tetrahedra not included in $T_i(M)$.
Once again, these moves are determined entirely by $\partial S_i \cap \partial M$.
\end{enumerate}
We term the boundary Pachner moves arising
from the above procedure the \emph{specified sequence} of boundary
Pachner moves.

Note that no $(3,1)$ boundary Pachner moves that attach a tetrahedron were performed.

\section{Overview of the proof of Mijatovi\'c's theorem \break (Seifert fibred case)}
\label{sec:mijatovic-sfs}

In this section, we give an outline of the proof of Theorem \ref{mijatovichaken}
in the case where $M$ is Seifert fibred, following \cite{alexsf}.

In \cite{alexsf}, Mijatovi\'c also dealt with many closed Seifert fibre spaces,
but we will not do so here. We will assume (as stated in Theorem \ref{mijatovichaken}) that
$M$ has non-empty boundary.

We will not consider here the case where the base orbifold of the Seifert fibration
has zero Euler characteristic.
In other words, we will exclude the case where $M$ is homeomorphic to an $I$-bundle
over a torus or Klein bottle. These spaces required a separate argument in \cite{alexsf}
because they admit properly embedded essential annuli which cannot be
made vertical in the Seifert fibration after an ambient isotopy.

As in the simple case, the goal is to build the triangulation ${\mathcal T}_{\rm can}$ of $M$
from a collection of surfaces. However, these surfaces do not exactly form
a hierarchy. They are as follows:
\begin{enumerate}
\item The first surface $S_1$ is a union of properly embedded disjoint tori, which bound a collection of solid
tori that together form a regular neighbourhood of the singular fibres. The exterior
of these solid tori, which we denote by $M_-$, is a union of regular fibres and hence a
circle bundle over a surface.
\item The second surface $S_2$ is a union of meridian discs, one for each of
the solid tori from (1). At this stage, Mijatovi\'c gives the tori $S_1$ a certain triangulation,
arising from these meridian discs and from the regular fibres in the Seifert
fibration. We will not dwell on the details of this triangulation, because
we will follow a slightly different approach in our proof of Theorem \ref{mijatovicmodified}.
\item The third surface is a horizontal section $S_3$ of the circle bundle $M_-$. Now, $M_-$ may
contain many horizontal sections, even up to ambient isotopy. The section $S_3$
is chosen so that its boundary is normal in the given triangulation of
$\partial M_-$ and has least weight among all such normal simple closed curves.
\item The fourth surface $S_4$ is a maximal collection of disjoint non-parallel vertical annuli properly embedded 
in $M_-$, each of which intersects $S_3$ in a single arc.
\end{enumerate}

Thus, this fails to be a hierarchy for two reasons. Firstly, $S_1$ is compressible in
$M$. Secondly, $S_4$ is not properly embedded in the
exterior of $S_3$. Nevertheless, $S_1 \cup S_2 \cup S_3 \cup S_4$ 
is a 2-complex, and the complementary regions of this complex
are balls. To construct the triangulation ${\mathcal T}_{\rm can}$, first a vertex is introduced
into the interior of each face of the 2-complex, and this face is coned off. Then a vertex is placed in the interior of each complementary ball,
and this too is coned off.

Just as in this previous section, it is important that this 2-complex is 
constructible, given an arbitrary triangulation $T$ of $M$ that restricts
to the given triangulation $T_{\partial M}$ on $\partial M$. In the previous section, we realised the first
surface as a normal surface in $T$ with bounded weight (as a function of the number of tetrahedra in $T$), and then
we performed a sequence of Pachner moves to $T$, taking it to
a triangulation in which the surface is simplicial.
In the case here, it is not immediately clear that the tori $S_1$
can be placed into normal form in $T$, because they are compressible. Instead,
one must construct them in stages. The first step is to find a maximal
collection of disjoint non-parallel essential vertical tori in the Seifert fibration
that are in normal form with respect to $T$. Then Pachner moves are
performed on $T$, after which these tori are simplicial. These tori decompose $M$ into
pieces, each of which is Seifert fibred. There are a limited range of possibilities
for these pieces. The base orbifold may be a pair of pants containing no singularities;
an annulus containing one singular point; a Mobius band with no singularities; or a disc with
two singularities. In each case, the piece contains one or two properly embedded
essential vertical annuli, which decompose the piece into one or two solid tori.
A subset of these solid tori are regular neighbourhoods of the singular
fibres, and the boundaries of these solid tori are the required surface $S_1$.

The rest of the argument follows the lines of the
simple case fairly closely. One new feature that arises in the Seifert fibred case is that 
it is not possible to isotope the section $S_3$ to a normal surface with
bounded weight. Instead, one must also use homeomorphisms of $M$ that
are supported in the interior of $M$. This is because of the presence of
incompressible normal tori. It is for this reason that the conclusions of
Theorems \ref{mijatovichaken} and \ref{mijatovicmodified} make reference to
a homeomorphism supported in the characteristic submanifold of $M$.

Another new feature in the Seifert fibred case is that
$S_4$ is not properly embedded in the exterior of $S_1 \cup S_2 \cup S_3$.
Instead, it is properly embedded in the exterior of $S_1 \cup S_2$. Thus,
one must ensure that the horizontal section $S_3$ intersects each vertical
annulus of $S_4$ in a single arc. This is clearly possible topologically,
but one must be careful to ensure that it holds, while at the same time
making $S_3$ and $S_4$ normal surfaces with bounded complexity.

The procedure in the Seifert fibred case for taking the given triangulation
$T$ to the canonical triangulation $\mathcal{T}_{\rm can}$ is very similar to
that used in Section \ref{sec:mijatovic-simple}. In particular, the
sequence of boundary Pachner moves that is used depends only on the
intersection between the curves $\partial S_1$, $\partial S_2$, $\partial S_3$,
$\partial S_4$ and $\partial M$.

\section{The construction of the canonical triangulation}
\label{sec:canonical-tri}

In this section, we consider a general compact orientable Haken 3-manifold $M$, satisfying the
hypotheses of Theorem \ref{mijatovicmodified}. In particular,
it may have JSJ tori. We start with a given triangulation $T_{\partial M}$
for $\partial M$. The goal is to define the canonical triangulation ${\mathcal T}'_{\rm can}$
for $M$ which agrees with $T_{\partial M}$ on $\partial M$. However, our first
step is to build a slightly simpler triangulation ${\mathcal T}_{\rm can}$, which does not
agree with $T_{\partial M}$ on $\partial M$. We use the terminology ${\mathcal T}_{\rm can}$
because in the case where $M$ is simple or Seifert fibred, it agrees with
Mijatovi\'c's triangulations described in the previous two sections.

Again, the aim is to build a 2-complex in $M$, for which each complementary region
is a 3-ball. From this, ${\mathcal T}_{\rm can}$ is built,
by coning off each face of the complex, and then coning off each complementary ball.
The 2-complex is, as in the previous sections, constructed from a collection of surfaces,
in a number of steps.

\vskip 6pt

\noindent {\bf Step 1.}
The first surface $S_1$ is built from the JSJ tori, which divide the manifold into
simple and Seifert fibred pieces. When a JSJ torus has simple pieces on both sides
(possibly the same simple piece), two parallel copies of the torus are used in $S_1$. Similarly,
when a JSJ torus has Seifert fibred pieces on both sides, two parallel copies of it are used.
However, when a torus has a simple piece on one side and a Seifert fibred piece on the
other, only one copy of the torus is used.

\vskip 6pt

\noindent {\bf Step 2.}
The next surfaces are hierarchies for the simple pieces, as described in 
Section \ref{sec:mijatovic-simple}.

\vskip 6pt

\noindent {\bf Step 3.}
When a JSJ torus has simple pieces on both sides, we have taken two copies of the torus.
Between these lies a region homeomorphic to $T^2 \times I$. The boundary of this region
has inherited a boundary pattern from the hierarchies in the adjacent pieces.
The next surface is a vertical annulus in each $T^2 \times I$ region, which intersects the boundary
pattern transversely and in as few points as possible. This cuts
$T^2 \times I$ into a solid torus. The next surface is a meridian disc for each such
solid torus, which again intersects the boundary pattern transversely and minimally.

\vskip 6pt

\noindent {\bf Step 4.}
We now turn to the JSJ tori that have Seifert fibred pieces on both sides.
We have taken two copies of each such torus, which bound a region homeomorphic to $T^2 \times I$.
We insert two vertical annuli into $T^2 \times I$, with respective slopes those of the regular fibres in the adjacent Seifert fibred pieces. We arrange for these
annuli to intersect transversely and minimally.

\vskip 6pt

\noindent {\bf Step 5.}
We now tackle the Seifert fibred pieces. For each such piece, its boundary components which do not lie
in the boundary of $M$ have inherited some boundary pattern,
which decompose the tori into discs. We now use a similar sequence of surfaces to that described in Section 9. At one stage, a slight
variant of the procedure is required. Recall that
in Step 3 in Section \ref{sec:mijatovic-sfs},
a section $S_3$ for the circle bundle $M_-$ is chosen, and this is required to have
minimal intersection number with the edges of the given triangulation of $\partial M_-$.
This triangulation of $\partial M_-$ was constructed in Step 2 of Section \ref{sec:mijatovic-sfs}. For those components
of $\partial M_-$ that are actually components of $\partial M$, they are just assigned
the given triangulation, which is the restriction of $T_{\partial M}$. 
For those components of $\partial M_-$ which bound
a solid torus neighbourhood of a singular fibre, a certain triangulation was constructed
using the regular fibres and a meridian disc of the solid torus. However, in our
case, we pursue a slightly different approach. We have already assigned a boundary pattern
to each component of $\partial M_- - \partial M$, which fills the surface. We pick the section
$S_3$ to have minimal intersection number with the union of this boundary pattern
and the 1-skeleton of $\partial M$.
In fact, the whole purpose of Step 4 above was to ensure that every component of
$\partial M_- - \partial M$ picks up a boundary pattern that chops up the torus into discs.

\vskip 6pt

The union of the above surfaces with $\partial M$ is a 2-complex. Once again, we
form ${\mathcal T}_{\rm can}$ from this 2-complex, by coning its faces and then
the complementary 3-balls.

Now, the restriction of ${\mathcal T}_{\rm can}$ to $\partial M$ clearly does not equal
$T_{\partial M}$. We now fix this, by introducing ${\mathcal T}'_{\rm can}$, which is the
canonical triangulation required by Theorem \ref{mijatovicmodified}.

Let $S_1, \dots, S_n$ be the above sequence of surfaces, described in Steps 1 - 5.
For $0 \leq i \leq n$, we will keep track of a triangulation $T_i(\partial M)$ for $\partial M$, in which
$\partial M \cap (S_1 \cup \dots \cup S_{i})$ is simplicial.
The initial triangulation $T_0(\partial M)$ will be the given triangulation
$T_{\partial M}$. 
Each surface $S_i$ intersects $\partial M$ in a collection of disjoint simple closed curves and arcs.
We realise these as normal curves and arcs in the triangulation $T_{i-1}(\partial M)$,
with the property that they intersect the 1-simplices in as few points as possible.
The new triangulation $T_i(\partial M)$ is chosen as follows. Each triangular face
of $T_{i-1}(\partial M)$ is divided up by arcs which are subsets of $\partial S_i$.
We declare that these arcs are 1-simplices in $T_i(\partial M)$. The regions in the complement of these arcs are discs.
A vertex is inserted into each such disc, and then the disc is coned off.
The result is $T_i(\partial M)$.

At this stage we make a specific choice of 2-dimensional Pachner
moves taking $T_{i-1}(\partial M)$ to $T_i(\partial M)$, as follows. 
A $(1,3)$ move is performed on each triangle of $T_{i-1}(\partial M)$.
Then, at each of the original edges $e$ of $T_{i-1}(\partial M)$, a (2,2)
move is performed. Then, at each such edge $e$, an alternating sequence
of (1,3) and (2,2) moves is performed. The number
of (1,3) moves is equal to the number of points of
intersection between $\partial S_i$ and $e$. The result of this is that
each of the original triangles of $T_{i-1}(\partial M)$ has
been transformed into a cone. (See Figure \ref{retriangulateboundary}.)  The triangulation
of this cone is now modified without changing its boundary
any further, using 2-dimensional Pachner moves, so that it becomes
$T_i(\partial M)$. We term this the \emph{specified
sequence} of Pachner moves taking $T_{i-1}(\partial M)$ to $T_i(\partial M)$.
In Section \ref{sec:mijatovic-simple},
a sequence of boundary Pachner moves on $M$, also called the \emph{specified sequence},
was defined. Note that the specified sequence of boundary Pachner moves
given in Section \ref{sec:mijatovic-simple} induces the specified sequence
of 2-dimensional Pachner moves defined here.


\begin{figure}[h]
\centering
\includegraphics{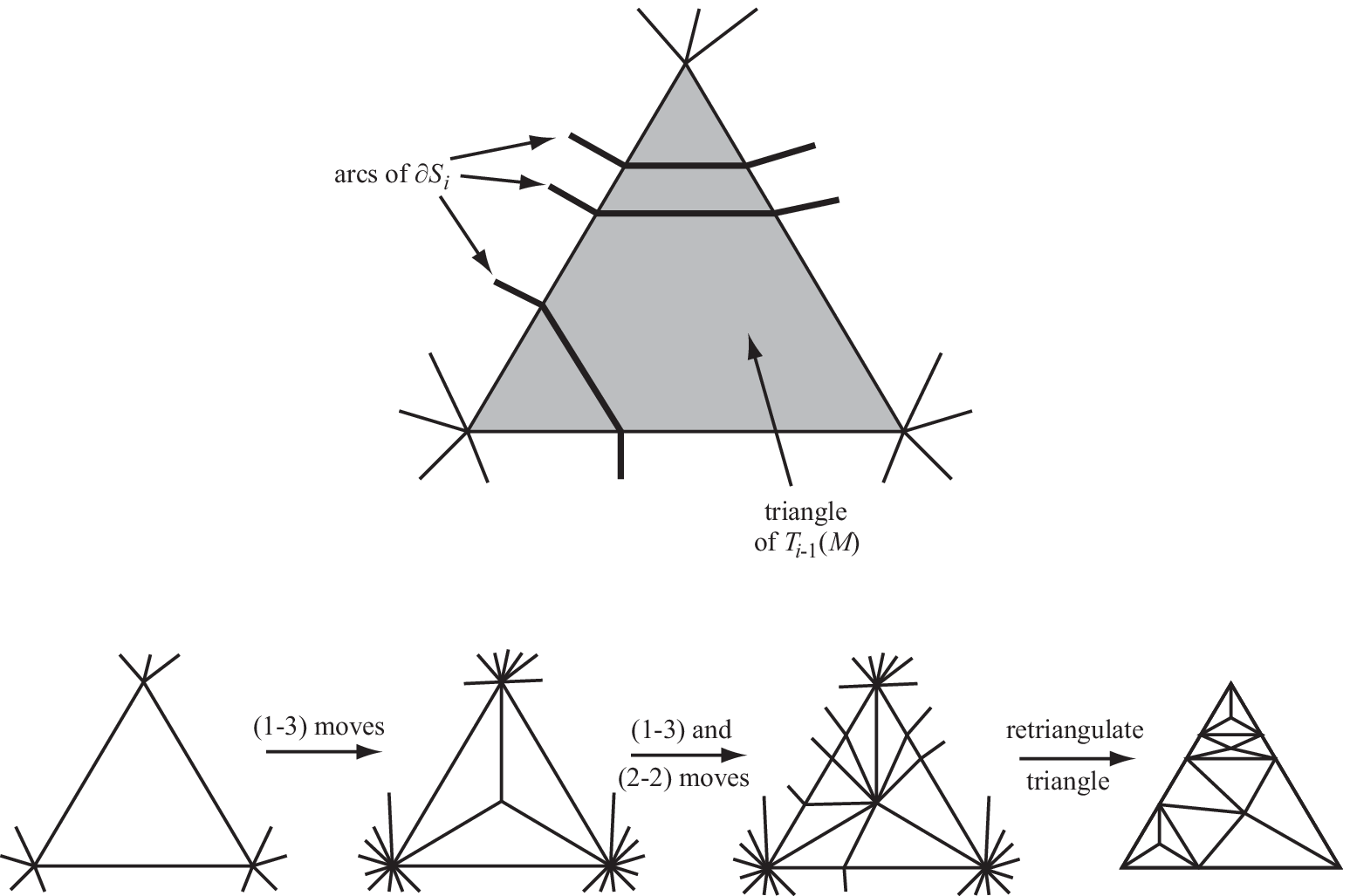}
\caption{} \label{retriangulateboundary}
\end{figure}

Associated to any sequence of 2-dimensional Pachner moves on a triangulated surface $F$,
there is a 3-dimensional space constructed as follows. One starts with $F$. Each
time a 2-dimensional move is performed, a 3-simplex is attached onto one side of $F$. 
Call this the `top' side. This top side is a copy of $F$ with its new triangulation. After a sequence
of these moves, the resulting space is approximately a copy of $F \times I$. The bottom
has the original triangulation of $F$, and the top has the new triangulation.
We say `approximately' $F \times I$, because this space need not in fact
be a 3-manifold. This is because it is possible that some simplices of $F$ are
left untouched by the sequence of 2-dimensional Pachner moves, in which
case these simplices lie in both the top and bottom copy of $F$.
We call this space a \emph{generalised product}. When $F$ is the boundary of 
a 3-manifold, we also call this space a \emph{generalised collar}.

We have defined above a specified sequence of 2-dimensional Pachner moves which
takes $T_{\partial M}$ to $T_n(\partial M)$. Associated with this sequence, there is
a generalised product. Now $T_n(\partial M)$ is the boundary triangulation
of ${\mathcal T}_{\rm can}$. Thus, we can attach the generalised product
to ${\mathcal T}_{\rm can}$, to create a triangulation of $M$ with boundary
triangulation $T_{\partial M}$. This is ${\mathcal T}_{\rm can}'$, 
the canonical triangulation for $M$.

\section{From boundary Pachner moves to interior ones} \label{sec:from-boundary-to-interior-moves}

Let $M$ be a compact orientable 3-manifold with a triangulation $T$.
Suppose that we are given a sequence of Pachner moves, starting with $T = T_0$
and ending with a triangulation $T_n$. We want to use this sequence to specify
a sequence of interior Pachner moves to $M$, giving triangulations $T = T_0', \dots, T_n'$
of $M$. At each stage, the triangulation $T_i'$ of $M$ will contain a triangulated
generalised collar on $\partial M$. If we form the closure of the complement
of this generalised collar, the result will be the triangulation $T_i$.

We say that such a sequence of interior Pachner moves and generalised collars are
\emph{associated} with the given sequence of Pachner moves.

The case we have in mind is where $T_n = \mathcal{T}_{\rm can}$.
A sequence of Pachner moves taking $T$ to $T_n = \mathcal{T}_{\rm can}$
arises from the proof of Theorem \ref{mijatovichaken}. The associated sequence
of interior Pachner moves will take $T$ to $T_n' = \mathcal{T}_{\rm can}'$,
and will be the moves required by Theorem \ref{mijatovicmodified}.

We start with the motivating case, where $T_i$ is obtained from $T_{i-1}$ by
a boundary Pachner move which attaches a simplex. Suppose, for simplicity, that
this has the effect of performing a $(1,3)$ move on the boundary. Now, embedded within
$T_{i-1}'$ is a copy of $T_{i-1}$. The boundary Pachner move attaches on a single
tetrahedron to a triangle in $T_{i-1}$. A copy of this triangle is in $T_{i-1}'$,
lying between $T_{i-1}$ and the generalised collar. What we do is insert two tetrahedra
into $T_{i-1}'$ at the location of this triangle. These two tetrahedra are glued to each other along three faces.
The tetrahedron that is adjacent to the generalised collar is included in the
generalised collar of $T_i'$. This procedure which inserts the two tetrahedra can
clearly be achieved by interior Pachner moves. One simply performs a $(1,4)$ move
on the tetrahedron of $T_{i-1}'$ which is adjacent to the triangle and which
is not part of the generalised collar. Then one performs a $(3,2)$ move.
This is illustrated in Figure \ref{bpach1-3}.

\begin{figure}
\centering
\includegraphics[width=\textwidth]{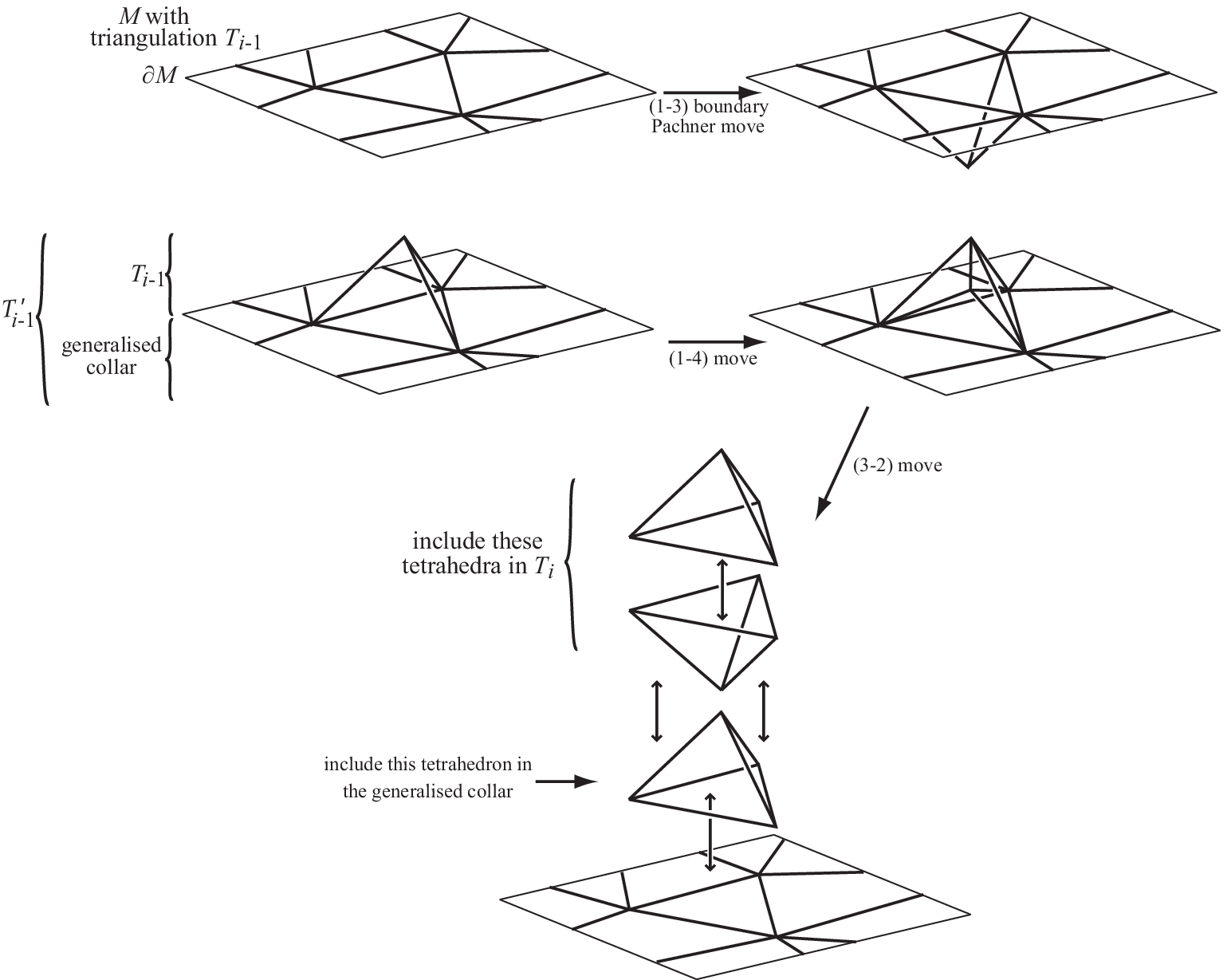}
\caption{} \label{bpach1-3}
\end{figure}

We now deal with some easy cases.

When an interior Pachner move is performed on $T_{i-1}$, we do the same
move to $T_{i-1}'$ and do not change the generalised collar.

When a boundary Pachner move is performed on $T_{i-1}$, and this removes
a tetrahedron, then we do not perform a Pachner move to $T_{i-1}'$. Instead, we
simply enlarge its generalised collar to include the relevant tetrahedron.

There are still two cases to consider. The triangulation $T_i$ may be obtained
from $T_{i-1}$ by a boundary Pachner move that adds a tetrahedron, and which
performs a $(2,2)$ move or a $(3,1)$ move on the boundary. We will not in fact
consider the $(3,1)$ case, because it is more complicated, and we will not need it
in this paper. So, suppose that a $(2,2)$ move is performed on the boundary of $T_{i-1}$.
Two triangles are involved in this move, and there are two corresponding
triangles in $T_{i-1}'$. These triangles are glued along an edge, and their
union is a disc (possibly with identifications along its boundary). 
What we do is `blow air' into this disc, creating two copies of the
disc. In the space between these two discs, we insert two tetrahedra. These two
tetrahedra are glued to each other along two triangles. The tetrahedron that is
adjacent to the generalised collar of $T_{i-1}'$ is included in the generalised
collar of $T_i'$. The other new tetrahedron is included in the copy of $T_i$
in $T_i'$.

We  need to explain how the insertion of these two tetrahedra, which takes
$T_{i-1}'$ to $T_i'$, can be achieved using interior Pachner moves. First apply a $(1,4)$ to every tetrahedron of $T_{i-1}'$. After this, the tetrahedra incident to the interior of any face are distinct.
Consider the two triangles involved in the boundary Pachner move. These two triangles are
adjacent to distinct tetrahedra because of the $(1,4)$ moves we have just applied. Consider these two tetrahedra. 
Apply some $(2,3)$ moves until they are adjacent. Now apply two $(2,3)$ moves in 
the resulting adjacent tetrahedra to insert the two tetrahedra in the place of 
the two triangles that get split open. This process is illustrated in Figure \ref{insert22}. 
Now apply $(3,2)$ moves and $(4,1)$ moves to undo the initial moves.
Note that if $T_i$ contains $t$ tetrahedra, then the number of interior Pachner moves
in this process is at most $4t$.

\begin{figure}
\centering
\includegraphics[width=0.7\textwidth]{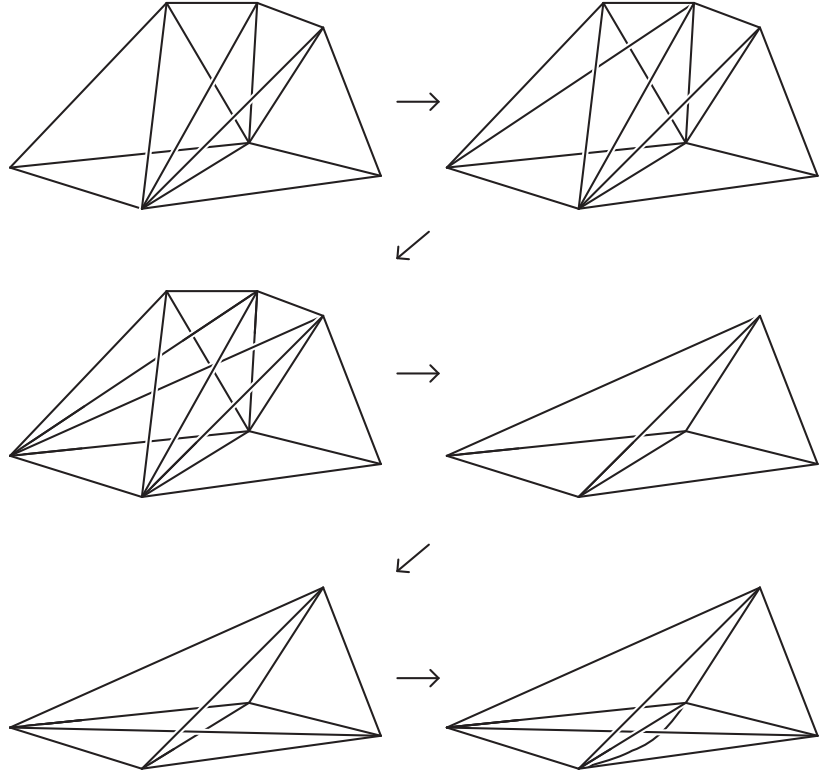}
\caption{A sequence of $(2,3)$ moves followed by a $(3,2)$ move} \label{insert22}
\end{figure}

We close with an important observation. The generalised collar constructed above,
which is a subset of $T_n'$, only depends on the boundary Pachner moves in the
original sequence.

Suppose that $T$ is a triangulation of $M$, and that its restriction to $\partial M$
is the given triangulation $T_{\partial M}$. Suppose that we have found a sequence
of Pachner moves which takes $T$ to ${\mathcal T}_{\rm can}$. Suppose also that the boundary Pachner moves in this sequence
precisely induce the specified sequence of 2-dimensional Pachner moves
on $\partial M$, defined in Section \ref{sec:canonical-tri}. Then we deduce that the associated
sequence of interior moves takes $T$ to ${\mathcal T}_{\rm can}'$. This important observation
will be crucial in our proof of Theorem \ref{mijatovicmodified}, which we now come to.

\section{Proof of Theorem \ref{mijatovicmodified}} \label{sec:proof-main}

We are given a triangulation $T$ for $\partial M$, and we need to use interior Pachner
moves, plus possibly a homeomorphism supported in the interior of $M$,
to take $T$ to ${\mathcal T}_{\rm can}'$. This will be achieved by performing
a sequence of Pachner moves that takes $T$ to ${\mathcal T}_{\rm can}$,
and so that the boundary Pachner moves in this sequence precisely induce the
specified sequence in Section \ref{sec:canonical-tri}. Then, as we observed
at the end of the previous section, the associated sequence of interior Pachner moves will take $T$ to
${\mathcal T}_{\rm can}'$, as required.

In Mijatovi\'c's proof and Section \ref{sec:mijatovic-simple}, a sequence of
Pachner moves was given, which takes $T$ to ${\mathcal T}_{\rm can}$.
But we need to ensure that the boundary Pachner moves in this section are
the specified sequence from Section  \ref{sec:canonical-tri}.

Let $S_1, \dots, S_n$ be the sequence of surfaces used in the definition of
${\mathcal T}_{\rm can}$. Suppose that we have subdivided the triangulation
$T$ to $T_{i-1}(M)$ using Pachner moves. Then $S_1, \dots, S_{i-1}$ are
simplicial in $T_{i-1}(M)$, and so $T_{i-1}(M)$ restricts to a triangulation
$T_{i-1}(M_{i-1})$ of $M_{i-1}$. We want to realise $S_i$ as
a normal surface in $T_{i-1}(M_{i-1})$ with bounded complexity.
The precise result that we use depends on what type of surface $S_i$ is.
For example, it may be the JSJ tori, or it may be a surface where $-\chi(S_i)$
is minimised, or it may be one of several other possibilities.
However, in each case, there is an ambient isotopy taking $S_i$ to a normal
surface with at most $2^{350t^2}$ triangles and squares,
where $t$ is the number of tetrahedra of $T_{i-1}(M_{i-1})$.
(See Lemma 4.5 in \cite{alexff}.)

However, there is a complication. In Section \ref{sec:canonical-tri}, we have
already specified the intersection between $S_i$ and $\partial M$. Recall that
we have already chosen the sequence of triangulations 
$T_{\partial M} = T_0(\partial M), T_1(\partial M), \dots, T_n(\partial M)$, where
the restriction of $T_{i-1}(M)$
to $\partial M$ is $T_{i-1}(\partial M)$. The simple closed curves and arcs
$S_i \cap \partial M$ need to be the specified normal arcs in $T_{i-1}(\partial M)$.
However, when the usual normalisation procedure is applied to $S_i$ in $T_{i-1}(M_{i-1})$, 
$\partial S_{i}$ may need to be moved.
We therefore require the following result, which we will prove in Section
\ref{sec:adjusting-boundaries}.



\begin{theorem}\label{fixnormal}
Let $M$ be a compact orientable irreducible 3-manifold with an essential boundary pattern $P$.
Let $T$ be a triangulation of $M$, in which $P$ is simplicial, and which
consists of $t$ tetrahedra. Let $F$ be a normal surface properly embedded in $M$, essential with respect to $P$, consisting of $n$ normal discs. Let $X$ be the closure of the union of some non-adjacent components of $\partial M - P$. Suppose that $\partial F \cap X$ is ambient isotopic to a collection of normal curves and arcs $c$ in $X$, and that $c$ has
minimal weight in its ambient isotopy class in $X$. Then $F$ is ambient isotopic in $M$ to a normal surface $F'$ whose boundary agrees with $c$ in $X$ and which consists of at most $8000n^4t^2$ normal discs.
\end{theorem}

Recall that $\partial S_i$ has been chosen so that its intersection with $\partial M$ minimizes weight in its ambient isotopy class. Let $c$ be $\partial S_i \cap \partial M$.
Now, according to Mijatovi\'c's arguments, $S_i$ can be isotoped in $M_i$ to
a normal surface with a bounded number of normal discs, $n$ say. This isotopy
may move $\partial S_i$. By the above result,
we may find a normal surface isotopic to $S_i$, with boundary that agrees with $S_i$ on $\partial M$,
and with at most $8000n^4t^2$ normal discs, where $t$ is the number of tetrahedra in $T_{i-1}(M_{i-1})$.

We now apply Pachner moves, as described in Section \ref{sec:mijatovic-simple},
taking $T_{i-1}(M)$ to $T_i(M)$. As explained in Section \ref{sec:mijatovic-simple},
the boundary Pachner moves that are used only depend on
the intersection between $\partial M$ and $S_{i}$. And we have ensured that these
are fixed as in Section \ref{sec:canonical-tri}. Thus, the boundary
Pachner moves precisely induce the specified sequence of 2-dimensional Pachner moves on $\partial M$ defined in Section 10. Hence, the
associated sequence of interior moves creates ${\mathcal T}_{\rm can}'$, as required.

We close this section by bounding the number of Pachner moves used.
To provide such a bound, it suffices to bound the number of moves required to
pass from $T_{i-1}'(M)$ to $T_i'(M)$, and also to bound the number $n$, where
$S_1, \dots, S_n$ is the sequence of surfaces used to define ${\mathcal T}_{\rm can}$. 
We start with the former estimate.

Suppose that $T_{i-1}(M)$ has $t$ tetrahedra. In \cite{alexff}, Mijatovi\'c argues that
at most $\exp^{(2)}(t)$ Pachner moves are required to pass to the next triangulation.
However, as we have seen, his triangulations and ours are a little different. Nevertheless,
we will now show that we also can pass from $T_{i-1}'(M)$ to $T_i'(M)$ 
using at most $\exp^{(2)}(t)$ interior Pachner moves.

Let us focus on the case where $S_i$ is a surface from Step 2. In other words,
suppose that it is part of a hierarchy in one of the simple pieces. 
Proposition 4.2 of \cite{alexknot} or Lemma 4.5 of \cite{alexff}
gives that $S_i$ can be realised as a normal surface in $T_{i-1}(M_{i-1})$
with at most $2^{350t^2}$ triangles and squares. However, because the intersection
$S_i \cap \partial M$ is specified in advance, then in fact we only get a bound of
at most $8000 \ 2^{1400t^2} t^2$ triangles and squares, using Theorem \ref{fixnormal}.
Then using Lemma 4.1 in \cite{alexsf}, there is a sequence of at most
$160000 \ 2^{1400t^2} t^3$ Pachner moves taking $T_{i-1}(M)$ to $T_i(M)$, 
in which $S_i$ is simplicial. However, boundary Pachner moves may be
used in this process, and so we need to bound the number of interior Pachner
moves required in the associated sequence that takes $T_{i-1}'(M)$ to
$T_i'(M)$. The boundary Pachner moves in this sequence are precisely 
those from the specified sequence described in Section \ref{sec:mijatovic-simple}. 
Recall that first a boundary $(1,3)$ move is performed on each triangle of
$T_{i-1}(M) \cap \partial M$. The number of such moves is at most $4t$, and hence the number of
interior moves in the associated sequence is at most $8t$.
Then a $(2,2)$ move is performed along each edge of $T_{i-1}(M) \cap \partial M$.
There are at most $6t$ such edges.
Using the bound at the end of Section \ref{sec:from-boundary-to-interior-moves},
the number of interior Pachner moves in the associated sequence is at
most $6t(4t)$. After this, the only boundary Pachner moves that are
performed are those that remove tetrahedra, and these do not create
any interior moves in the associated sequence. So, the number of interior Pachner moves
taking $T_{i-1}'(M)$ to $T_i'(M)$ is at most
$$8t + 6t(4t) + 160000 \ 2^{1400t^2} t^3 \leq \exp^{(2)}(t),$$
as claimed. Note that the inequality holds because the values of $t$ for which we are making this estimate are sufficiently large. The number of tetrahedra in $T_i'(M)$ is also at most
$\exp^{(2)}(t)$.

An easy induction gives that, for $i \geq 1$, the number of tetrahedra
in $T_i'(M)$ is at most $\exp^{(2i)}(t)$, and that the number of Pachner
moves required to take $T_{i-1}'(M)$ to $T_i'(M)$ is also at most
$\exp^{(2i)}(t)$.

Thus, all we need to do now is bound the number $n$ of surfaces in the sequence $S_1, \dots, S_n$.
In Step 1, only one surface (which may be disconnected) is used. The number of surfaces in
Step 2 is at most $2^{160t}$, by a bound of Mijatovi\'c in Section 5 of \cite{alexff}.
In Step 3, two surfaces are used. In Step 4, two surfaces are also used.  The final step
is Step 5, which deals with the Seifert fibred pieces. Only four surfaces are used, but
in fact the first surface should be counted as two, because it is built in two stages.
So, we can certainly take $n \leq 2^{161t}$.

Thus, the number of interior Pachner moves taking $T$ to ${\mathcal T}_{\rm can}'$ is
at most $\exp^{(a^t)}(t)$, where $a = 2^{162}$. This proves Theorem \ref{mijatovicmodified}.

\section{Adjusting boundaries of normal surfaces}\label{sec:adjusting-boundaries}

In this section we prove Theorem \ref{fixnormal}. Our strategy will be to isotope $F$ so that its boundary is in the correct place and has controlled edge degree. Then we shall apply to the resulting surface a standard normalization procedure that doesn't increase edge degree and which, by the hypotheses on $c$, does not affect the intersection with $X$. The bound on the edge degree of this surface will yield the required bound on the number of normal discs in $F'$.

Observe that $\partial F \cap X$ certainly consists of at most $4n$ normal arcs. Hence $c$ consists of at most $4n$ normal arcs. Without loss of generality pick $c$ so that every normal arc of $c$ intersects each normal arc of $\partial F$ in at most one point. Then the number of intersection points between $\partial F$ and $c$ is at most $(4n)^2=16n^2$. Label the components of $\partial F \cap X$ as $\{\alpha_1,\ldots,\alpha_m\}$ and the components of $c$ as $\{\beta_1,\ldots,\beta_m\}$ so that there is an ambient isotopy of $X$ that takes $\alpha_i$ to $\beta_i$ for each $i$. Consider $\alpha_1$ and $\beta_1$. These intersect in at most $16n^2$ points. We will now isotope $F$ about in a small collar neighborhood of $X$ to obtain a new surface where the image of $\alpha_1$ is equal to $\beta_1$. This is achieved as follows. Note that for convenience we will refer to the surfaces obtained from $F$ by our isotopies as $F$, with the same notation for the components of $\partial F \cap X$.
 By Lemma 3.1 of \cite{intcurve} applied to the double of $X$ along its boundary, either $\alpha_1$ and $\beta_1$ are disjoint, or there is a bigon of $\alpha_1 \cup \beta_1$ in the interior of $X$ that bounds a disc $B$ whose boundary consists of a subarc of $\alpha_1$ and a subarc of $\beta_1$ and whose interior is disjoint from $\alpha_1 \cup \beta_1$, or there is a triangle, $B'$, in $X$ whose boundary consists of a subarc of $\partial X$,  a subarc of $\alpha_1$ and a subarc of $\beta_1$  and whose interior is disjoint from $\alpha_1 \cup \beta_1$. In the latter two cases, we would like to isotope $\alpha_1$ across $B$ or $B'$, but we are impeded by the fact that there may be arcs of $\partial F$ other than $\alpha_1$ running though $B$. Pick one of these arcs that is outermost on $B$ or $B'$ and label the disc it cuts off $G$. This is a bigon or triangle whose interior is disjoint from $\partial F \cup c$ and whose boundary consists of a subarc, $\sigma$, of $\alpha_i$, for some $i$, a subarc, $\tau$, of $\beta_1$ and possibly a subarc, $\gamma$, of $\partial X$. Now isotope $F$ by sliding $\alpha_i$ across $G$ in a small collar neighborhood of $\partial M$ to form a new surface. More precisely, choose a product structure $X \times I$, where $I = [0,1]$, on a small regular neighborhood of $X$ in $M$ so that $(\partial F \cap X)\times \{0\} = \partial F \cap X$, the components of $F \cap (X \times I)$ are vertical in $X \times I$ and the arcs of 1-skeleton emanating away from $X$ into the interior of $M$ or parts of $\partial M$ outside $X$ are at most a small deviation away from being vertical. Also suppose that $(\partial X \times I) \cap \partial M = \partial X \times I$. The new surface obtained by sliding $F$ across $G$ is then obtained by removing $\sigma \times I$, inserting $(G \times \{1\}) \cup (\tau \times I)$ and then pushing the resulting surface a little further to make it disjoint from $G$. Note that performing this operation increases the edge degree of $F$ by at most $12t + 4n$, the first term being an upper bound on the number of possible new intersections of $F$ with the arcs of 1-skeleton not entirely in $X$, and the second term being an upper bound on the number of new intersections of $F$ with arcs entirely in $X$. Repeat this operation to remove all arcs of intersection of $B$ or $B'$ with $\partial F$ and then slide $F$ across $B$ or $B'$ itself. Repeat this proceed until $\alpha_1$ and $\beta_1$ are disjoint. The edge degree has been increased by at most $16n^2(12t +4n)$. Now, $F$ has been arranged so that $\alpha_1$ is disjoint from $\beta_1$. Hence $\alpha_1$ and $\beta_1$ cobound an annulus or, together with two arcs of $\partial X$, a disc. Call this annulus or disc $A$. We wish to isotope $\alpha_1$ across $A$ to make $\alpha_1$ equal to $\beta_1$. This time we may be impeded by arcs of intersection of $\partial F$ with $A$, which may be removed as before, or by entire curves or arcs of intersection of $\partial F$ with $A$, which may be removed by sliding across annuli or discs in a fashion similar to sliding across bigons and triangles. The effect of the removal of bigons and triangles on edge degree is already incorporated in our previous estimate. Sliding across an annulus or disc increases edge degree by at most $12t + 4n$ and we need to do this at most $m-1 \leq 4n-1$ times before $\partial F$ is disjoint from the interior of $A$. The final isotopy we perform at this stage is across $A$ itself, except we do not make the final push away from $A$. Thus the edge degree since our last estimate has increased by at most $4n(12t + 4n)$. To recap, we have isotoped $F$ about so that $\alpha_1=\beta_1$ and so that the edge degree of $F$ has increased by at most $(16n^2 + 4n)(12t + 4n)$. 

We may now repeat the above procedure with the other components of $\partial F \cap X$, treating them one at a time and keeping components fixed once they have been isotoped to equal the corresponding component of $c$. Thus we form a surface whose boundary components intersect $X$ in a way that agrees with $c$ and whose edge degree increases overall by at most $m(16n^2 + 4n)(12t + 4n) \leq 4n(16n^2 + 4n)(12t + 4n)$. The initial surface $F$, before any isotopies were performed, had edge degree at most $4n$. Hence the total edge degree of our surface is at most $4n + 4n(16n^2 + 4n)(12t + 4n) \leq 3844n^4t$. Now apply the normalization procedure as in Theorem 3.3.21 of \cite{mat}. Our hypotheses mean that this procedure does not affect the isotopy class of $F$, nor does it affect the intersection of $F$ with $X$. Further, edge degree does not increase. The resulting surface is $F'$. The valence of an edge in $T$ is at most $6t$ and so the number of normal discs in $F'$ is certainly at most $\frac{6t}{3}3844n^4t < 8000n^4t^2$. \hfill $\square$

\bibliographystyle{amsplain}
\bibliography{rmrefs}

\end{document}